\documentclass{amsart}
\usepackage{amsmath} 
\usepackage{amssymb}

\newtheorem{theorem}{Theorem}[section] 

\newtheorem{lemma}[theorem]{Lemma} 
\newtheorem{proposition}[theorem]{Proposition} 
\newtheorem{observation}[theorem]{Observation} 
 
\newtheorem{crlem}[theorem]{Crucial Lemma}
\newtheorem{crcor}[theorem]{Crucial Corollary}

\theoremstyle{definition}
\newtheorem{definition}[theorem]{Definition}

\newtheorem{discussion}[theorem]{Discussion}

\theoremstyle{remark}
\newtheorem{remark}[theorem]{Remark}
\newtheorem{notation}[theorem]{Notation}

\newtheorem{question}[theorem]{Question}

\newcommand{\cA}{{\mathcal A}}
\newcommand{\bA}{{\mathbf A}}
\newcommand{\bB}{{\mathbf B}}
\newcommand{\cB}{{\mathcal B}}

\newcommand{\gc}{{\mathfrak c}}
\newcommand{\gd}{{\mathfrak d}}
\newcommand{\cF}{{\mathcal F}}
\newcommand{\bG}{{\mathbf G}}
\newcommand{\cG}{{\mathcal G}}
\newcommand{\cH}{{\mathcal H}}
\newcommand{\bH}{{\mathbf H}}
\newcommand{\bh}{{\mathbf h}}
\newcommand{\cI}{{\mathcal I}}
\newcommand{\bi}{{\mathbf i}}

\newcommand{\bj}{{\mathbf j}}
\newcommand{\bbL}{{\mathbb L}}
\newcommand{\bbN}{{\mathbb N}}
\newcommand{\cP}{{\mathcal P}}
\newcommand{\bbP}{{\mathbb P}}
\newcommand{\bbQ}{{\mathbb Q}}

\newcommand{\bS}{{\mathbf S}}
\newcommand{\bT}{{\mathbf T}}
\newcommand{\cT}{{\mathcal T}}
\newcommand{\bV}{{\mathbf V}}
\newcommand{\cY}{{\mathcal Y}}

\newcommand{\Th}{{\rm Th}}
\newcommand{\PA}{{\rm PA}}
\newcommand{\pr}{{\rm pr}}
\newcommand{\BA}{{\rm BA}}
\newcommand{\OP}{{\rm OP}}
\newcommand{\OB}{{\rm OB}}
\newcommand{\otp}{{\rm otp}}
\newcommand{\Per}{{\rm Per}}
\newcommand{\arcl}{\mbox{\rm ar-cl}}
\newcommand{\rest}{{\restriction}}
\newcommand{\nor}{{\rm nor}}
\newcommand{\val}{{\rm val}}
\newcommand{\dis}{{\rm dis}}
\newcommand{\set}{{\rm set}}
\newcommand{\pos}{{\rm pos}}
\newcommand{\vpos}{{\rm vpos}}
\newcommand{\wpos}{{\rm wpos}}
\newcommand{\xpos}{{\rm xpos}}
\newcommand{\ypos}{{\rm ypos}}
\newcommand{\suc}{{\rm suc}}
\newcommand{\CR}{{\rm CR}}
\newcommand{\supp}{{\rm supp}}
 
\newcommand{\lh}{\ell g} 

\newcommand{\conc}{{}^\frown\!}
\newcommand{\proj}{{\rm proj}}

\newcount\skewfactor
\def\mathunderaccent#1#2 {\let\theaccent#1\skewfactor#2
\mathpalette\putaccentunder}
\def\putaccentunder#1#2{\oalign{$#1#2$\crcr\hidewidth
\vbox to.2ex{\hbox{$#1\skew\skewfactor\theaccent{}$}\vss}\hidewidth}}
\def\name{\mathunderaccent\tilde-3 }

\begin{document}


\title[Models of Expansions of ${\mathbb N}$]{Models of Expansions of
  ${\mathbb N}$ with no end extensions}  
\author{Saharon Shelah}
\address{Einstein Institute of Mathematics\\
Edmond J. Safra Campus, Givat Ram\\
The Hebrew University of Jerusalem\\
Jerusalem, 91904, Israel\\
 and  Department of Mathematics\\
 Rutgers University\\
 New Brunswick, NJ 08854, USA}
\email{shelah@math.huji.ac.il}
\urladdr{http://shelah.logic.at}

\thanks{I would like to thank Alice Leonhardt for the beautiful
typing. The author acknowledges support from the United States-Israel
Binational Science Foundation (Grant no. 2002323). Publication 937.}

\begin{abstract}  
  We deal with models of Peano arithmetic (specifically with a question of
  Ali Enayat). The methods are from creature forcing.  We find an expansion
  of $\bbN$ such that its theory has models with no (elementary) end
  extensions.  In fact there is a Borel uncountable set of subsets of $\bbN$
  such that expanding $\bbN$ by any uncountably many of them suffice. 
  Also we find arithmetically closed $\cA$ with no ultrafilter on it
  with suitable definability demand (related to being Ramsey).
\end{abstract}

\maketitle

\section{Introduction}
Recently, solving a long standing problem on models of Peano arithmetic,
(appearing as Problem 7 in the book \cite{KoSc06}), Ali Enayat proved (and
other results as well): 

\begin{theorem}
\label{q0.1} 
[See \cite{Ena08}]
For some arithmetically closed family $\cA$ of subsets of $\omega$, the
model $\bbN_{\cA}=(\bbN,A)_{A \in {\cA}}$ has no conservative extension
(i.e., one in which the intersection of any definable subset with $\bbN$
belongs to ${\cA}$). 
\end{theorem}
Motivated by this result he asked:
\begin{question}
\label{q0.4}
Is there $\cA\subseteq\cP(\omega)$ such that some model of $\Th(\bbN_{\cA})$
has no elementary end extension?
\end{question}
This asks whether the countability demand in the MacDowell-Specker theorem
is necessary. This classical theorem says that if $T$ is a theory in a
countable vocabulary $\tau = \tau_T$ extending $\tau(\bbN)=\{0,1,+,\times\}$
and $T$ contains $\PA(\tau)$, then any model of $T$ has an (elementary) end
extension; Gaifman continues this theorem in several ways, e.g., having
minimal extensions (see \cite{KoSc06} on it). The author \cite{Sh:66}
continues it in another way: we do not need addition and multiplication,
i.e., any model of $T$ has an elementary end extension when $\tau$ is a
countable vocabulary, $\{0,<\} \subseteq \tau$, $T$ is a (first order)
theory in $\bbL(\tau)$, $T$ says that $<$ is a linear order with 0 first,
every element $x$ has a successor $S(x)$, and all cases of the induction
scheme belong to $T$.

Mills \cite{Mil78} prove that there is a countable non-standard model of PA
with uncountable vocabulary such that it has no elementary end extension.

We answer the question \ref{q0.4} positively in \S4, we give a sufficient
condition in \S2 and deal with a relevant forcing in \S3. In fact we get an
uncountable Borel set $\bB\subseteq {\cP}(\bbN)$ such that if $B_\alpha\in
\bB$ for $\alpha<\alpha_*$ are pairwise distinct and $\alpha_*$ is
uncountable, then $\Th(\bbN,B_\alpha)_{\alpha<\alpha_*}$ satisfies the
conclusion. 

Enayat \cite{Ena08} also asked:
\begin{question}
\label{q0.7} 
Can we prove in ZFC that there is an arithmetically closed ${\cA}\subseteq
{\cP}(\omega)$ such that ${\cA}$ carries no minimal ultrafilter?
\end{question}
He proved it for the stronger notion of {\em 2-Ramsey ultrafilter}.  We hope
to deal with the problem later (see \cite{Sh:944}); here we prove that there
is an arithmetically closed Borel set $\bB\subseteq {\cP}(\bbN)$ such that
any expansion $\bbN$ by any uncountably many members of $\bB$ has such a
property, i.e., the family of definable subsets of $\bbN^+$ carry no
{\em 2.5-Ramsey ultrafilter}. 

Note that
\begin{enumerate}
\item[$(*)$]  if $N \ne \bbN$ is a model of PA which has no cofinal minimal
  extension, then on ${\rm StSy}(N)$ there is no  minimal ultrafilter, see
  Definitions \ref{0z.2}, \ref{0z.7}(1). 
\end{enumerate}

Enayat also asks:
\begin{question}
\label{Q2}
For a Borel set $\cA\subseteq {\cP}(\omega)$:
\begin{enumerate}
\item[(a)]  does the model $\bbN_{\cA}$ have a conservative end
  extension?  This is what is answered here (in the light of the previous  
  paragraph). 
\item[(b)] Suppose further that $\cA$ is arithmetically closed. Is $(\cA
  \cap [\omega]^{\aleph_0},\supseteq)$ a proper forcing notion? 
\end{enumerate}
\end{question}
The results here solve \ref{Q2}(a) and the second, \ref{Q2}(b), is solved in 
Enayat-Shelah \cite{EnSh:936}.

Enayat suggests that if we succeed to combine an example for ``${\rm
  StSy}(N)$ has no minimal ultrafilter'' and Kaufman-Schmerl \cite{KaSc84}, 
then we shall solve the ``there is $N$ with no cofinal minimal extension''
(Problem 2 of \cite{KoSc06}).

Note that our claim on the creature forcing gives suitable kinds of Ramsey
theorems.

We thank the audience for comments in the lectures given on it in the 
Rutgers Logic Seminar (October 2007) and later in Jerusalem's Logic Seminar
and we would like to thank the referee for pointing out some gaps, many
corrections and help.

\bigskip
\centerline {$* \qquad * \qquad *$}
\bigskip

\begin{notation}
\label{0z.1} 
\begin{enumerate}
\item As usual in set theory, $\omega$ is the set of natural numbers.
  Let $\pr:\omega\times \omega\longrightarrow \omega$ be the standard
  pairing function (i.e., pr$(n,m) = \binom{n+m}{2} +n$, so one--to--one
  onto two--place function). 
\item Let $\cA$ denote a subset of $\cP(\omega)$.
\item The Boolean algebra generated by $\cA\cup [\omega]^{< \aleph_0}$ will
  be denoted by $\BA(\cA)$.
\item Let $D$ denote a non-principal ultrafilter on $\cA$. When $\cA$ is not
  a sub-Boolean-Algebra of $\cP(\omega)$, this means that $D \subseteq \cA$
  and there is a unique non-principal ultrafilter $D'$ on the Boolean
  algebra $\BA(\cA)$ such that $D=D'\cap\cA$. (In \ref{0z.7} this extension
  makes a difference.)
\item Let $\tau$ denote a vocabulary extending $\tau_{\PA}=\tau_{\bbN} =
  \{0,1,+,\times,<\}$, usually countable. 
\item $\PA_\tau=\PA(\tau)$ is Peano arithmetic for the vocabulary $\tau$. 
\item A model $N$ of $\PA(\tau)$ is {\em ordinary\/} if $N \rest
\tau_{\PA}$ extends $\bbN$; usually our models will be ordinary.
\item $\varphi(N,\bar{a})$ is $\{b:N \models \varphi[b,\bar{a}]\}$, where
  $\varphi(x,\bar{y})\in\bbL(\tau_N)$ and $\bar{a}\in{}^{\ell g(\bar{y})}N$.   
\item $\Per(A)$ is the set (or group) of permutations of $A$. 
\item For sets $u,v$ of ordinals let $\OP_{v,u}$, ``the order preserved
   function from $u$ to $v$'', be defined by: 

$\OP_{v,u}(\alpha) = \beta$ if and only if

$\beta\in v$, $\alpha\in u$ and $\otp(v\cap\beta)=\otp(u\cap\alpha)$. 
\item We say that $u,v \subseteq {\rm Ord}$ form a $\Delta$--system pair
  when $\otp(u)=\otp(v)$ and $\OP_{v,u}$ is the identity on $u\cap v$.
\end{enumerate}
\end{notation}

\begin{definition}
\label{0z.2} 
\begin{enumerate}
\item For $\cA\subseteq\cP(\omega)$ we let 
\[\hspace{-10pt}\arcl(\cA) = \{B \subseteq \omega:B\mbox{ is first order 
definable in }(\bbN,A_1,\ldots,A_n)\mbox{ for some }n<\omega \mbox{ and }
A_1,\ldots,A_n\in\cA\}.\]
The set $\arcl(\cA)$ is called {\em the arithmetic closure of $\cA$}. 
\item For a model $N$ of $\PA(\tau)$ let {\em the standard system of $N$} be  
\[{\rm StSy}(N)=\{\varphi(N',\bar{a})\cap\bbN:\varphi(x,\bar{y})\in
\bbL(\tau) \mbox{ and }\bar{a}\in {}^{\ell g(\bar{y})}N\}\]
for any ordinary model $N'$ isomorphic to $N$. 
\end{enumerate}
\end{definition}

\begin{definition}
\label{0z.7}
\begin{enumerate}
Let $\cA\subseteq \cP(\omega)$.
\item For $h \in {}^\omega \omega$ let ${\rm cd}(h) = \{\pr(n,h(n)):
  n<\omega\}$, where $\pr$ is the standard pairing function of $\omega$, see 
  \ref{0z.1}(1). 
\item An ultrafilter $D$ on $\cA$, is called {\em minimal\/} when: 

if $h \in {}^\omega \omega$ and ${\rm cd}(h)\in\cA$, then for some $X \in D$
we have that $h \rest X$ is either constant or one-to-one.
\item An ultrafilter $D$ on $\cA$ is called {\em Ramsey\/} when: 

if $k<\omega$ and $h:[\omega]^k\longrightarrow \{0,1\}$ and ${\rm cd}(h) \in
\cA$, then for some $X\in D$ we have $h\rest [X]^k$ is constant.  

Similarly we define $k$-Ramsey ultrafilters.
\item $D$ is called {\em 2.5-Ramsey\/} or {\em self-definably closed\/}
  when:  

if $\bar{h}=\langle h_i:i<\omega\rangle$ and $h_i\in {}^\omega(i+1)$ and
${\rm cd}(\bar{h})=\{\pr(i,\pr(n,h_i(n)):i< \omega,n <\omega\}$ belongs to
$\cA$, then for some $g \in {}^\omega \omega$ we have: 

\[{\rm cd}(g)\in\cA\mbox{ and }(\forall i)[g(i)\le i \wedge \{n <
\omega:h_i(n) = g(i)\} \in D];\]
this follows from 3-Ramsey and implies 2-Ramsey. 
\item $D$ is {\em weakly definably closed\/} when: 

if $\langle A_i:i<\omega\rangle$ is a sequence of subsets of $\omega$ and
$\{\pr(n,i):n \in A_i$ and $i <\omega\}\in\cA$, then $\{i:A_i\in D\}\in
\cA$, (follows from 2-Ramsey); Kirby called it ``definable'';  Enayat uses
``iterable''. 
\end{enumerate}
\end{definition}

\begin{definition}
\label{0z.9} 
For $\cA\subseteq \cP(\omega)$ let $\bbN_{\cA}$ be $\bbN$ expanded by a
unary relation $A$ for every $A\in\cA$, so formally it is a
$\tau_{\cA}$--model, $\tau_\cA=\tau_\bbN\cup\{P_A:A\in \cA\}$, but below if
we use $\cA=\{A_t:t\in X\}$, then we actually use $\{P_t:t \in X\}$.
\end{definition}

\begin{definition}
\label{0z.13} 
Let $N$ be a model of $T\supseteq\PA(\tau)$, $\tau = \tau_T$.
\begin{enumerate}
\item We say that $N^+$ is an end extension of $N$ when:
\begin{enumerate}
\item[(a)] $N \prec N^+$, 
\item[(b)] if $a\in N$ and $b\in N^+\setminus N$, then $N^+ \models a < b$. 
\end{enumerate}
\item  We say $N^+$ is a conservative [end] extension of $N$ whenever
(a),(b) hold and
\begin{enumerate}
\item[(c)] if $\varphi(x,\bar{y})\in\bbL(\tau)$, $\bar{b} \in {}^{\ell
    g(\bar{y})} (N^+)$, then $\varphi(N^+,\bar b)\cap N$ is a definable
  subset of $N$. 
\end{enumerate}
\end{enumerate}
\end{definition}

\bigskip
\centerline {$* \qquad * \qquad *$}
\bigskip

\begin{discussion}
\label{0z.19}
We may ask: How is the {\em creature forcing} relevant?  Do we need
Roslanowski--Shelah \cite{RoSh:470}?

The creatures (and creatures forcing) we deal with fit \cite{RoSh:470}, but
instead of CS iteration it suffices for us to use a watered down version of
creature iteration. That is here it is enough to define $\bbQ_u$ for finite
$u \subseteq{\rm Ord}$ such that:
\begin{enumerate}
\item[(a)$_1$]  $\bbQ_u$ is a creature forcing with generic $\langle
  {\name{t}_\alpha}:\alpha \in u\rangle$; this restriction implies that
  cases irrelevant in full forcing where we have to use countable $u$, are
  of interest here; hence we can use creature forcing rather than iterated
  creature forcing. 
\item[(a)$_2$]  In \S3, $\bbQ_u$ is a good enough ${}^\omega
  \omega$--bounding creature forcing, so we have continuous reading of
  names. 
\item[(a)$_3$]  We are used to do it above a countable models $N$ of
  ZFC$^-$, and this seems more transparent. But actually asking on the
  $\Delta_n$--type of the generic over $\bbN$ suffices. That is, we can,
  e.g., by $\Delta_{n+7}$ formula over $\bbN$ find, e.g., a condition
  $p\in\bbQ_u$ such that any $\bar{t}\in\bB_p$, e.g. a branch in the tree
  its $\Delta_n$-type over $\bbN$, i.e. the $\Delta_n$-theory of
  $(\bbN,\bar{t})$, so $t_\ell$ acts as a predicate (we can think of $\bB_u$
  as $\subseteq {}^u({}^\omega 2))$.
\end{enumerate}
Here the construction is by forcing over a countable $N_* \prec
(\cH(\chi),\in)$. Note that there is no problem to add $\cA^*:= N_*\cap
\cP(\omega)$. So we can prove the results for $\cA =$ (countable) $\cup$
(perfect). To improve it to perfect we need to force for PA by induction on
$n$ for $\Sigma_n$ formulas. 
\begin{enumerate}
\item[(a)$_4$] Note: for this it is O.K. if in every $p\in\bbQ_u$ the total
  number of commitments of the form ``$\rho$ is a member of $\varrho_x(i)$''
  is finite.  
\item[(b)$_1$] We can use $u_n = {}^n 2$, just a notational change, we would
  like to choose $p_n$ by induction on $n < \omega$ such that:
\begin{enumerate}
\item[$(\alpha)$]  $p_n\in\bbQ_{u_2}$, 
\item[$(\beta)$]  $p_n$ is such that for $\bar{t}\in\bB_{p_n}$ the
  $\Sigma_n$--theory of $(\bbN,\bar{t})$ can be read continuously on $p$,
\item[$(\gamma)$]  if $h:{}^n2\longrightarrow {}^{n+1}2$ is such that
  $(\forall\rho \in {}^n2)(h(\rho) \rest n = \rho)$, then $h(p_n) = p_n
\rest{\rm Rang}(h)$ both defined naturally (can make one duplicating at a
time). 
\end{enumerate}
\item[(b)$_2$]  In (b)$_1$, the set $\bigcup\{\varrho_x(i):x \in p\}$ grows
  from $p_n$ to $p_{n+1}$, i.e., here we need the major point in the choice
  of $\nor^0_x(C)$; however we do not need to diagonalize over it as in the
  proof about $\bbQ_u$. 
\item[(c)$_1$]  However, in \S3 we can define full creature iterated
  forcing, i.e. using countable support; it is of interest but irrelevant
  here;
\item[(c)$_2$]  but some cases of such creature forcing may look like: look
  at 
\[\bT' =\bigcup\{\prod_{k<n}(i+1):n < \omega\},\]
and the ideal
\[\{A\subseteq\prod_{i<\omega}(i+1):A = \bigcup_{n<\omega} A_n \mbox{ and } 
(\forall n<\omega)(\forall \eta\in\bT')(\exists\nu \in \suc_{\bT'}(\eta))
(\forall \eta \in A_n)[\neg(\nu\triangleleft\eta)]\}.\]
\item[(c)$_3$]  In the cases in which (c)$_2$ is relevant, we get a Borel
  set $\bB$ such that $(\bbN,t)_{t\in\bB}\ldots$, but not ``for every
  $\aleph_1$--members of $\bB$ we have$\ldots$''. 
\item[(d)] Actually, what we use are iterated creature forcing, but as we
  deal only with $\bbQ_u$, $u$ finite, so here we need not rely on the
  theory of creature iteration. 
\end{enumerate}
\end{discussion}

\section{Models of theories of expansions of $\bbN$ with no end extensions}  

\begin{theorem}
\label{4d.1} 
\begin{enumerate}
\item For some $\cA\subseteq \cP(\omega)$ some model of $\Th(\bbN_\cA)$ has
  no end extension. 
\item There is an uncountable Borel set $\cA\subseteq\cP(\omega)$ such that
  for any uncountable $\cA'\subseteq\cA$ the theory $T:=\Th(\bbN_{\cA'})$
  has a model with no end extension. 
\item In fact, any model $N$ of $T$ such that the naturally associated tree
  (set of levels $N$, the set of nodes of level $n \in N$ is $({}^n 2)^N$)
  has no undefinable branch is O.K.; such models exist by \cite{Sh:73}.
\item Moreover, without loss of generality, the set of subsets of $\bbN$
  definable in $\bbN_\cA$ is Borel. 
\end{enumerate}
\end{theorem}
\bigskip

The proof is broken to a series of definitions and claims finding a
sufficient condition proved in Sections 2, 3. More specifically,
Theorem \ref{4d.13}(b) gives a sufficient condition which is proved in
Proposition \ref{3c.37}.  

\begin{definition}
\label{4d.3}
\begin{enumerate}
\item Let sequences $\bar{n}^*=\langle n^*_i:i<\omega\rangle$ and
  $\bar{k}^*= \langle k^*_i:i<\omega\rangle$ be such that $n_0^* = 0$,
  $n^*_i \ll k^*_{i+1} \ll n^*_{i+1}$ for $i < \omega$. We can demand that
  the ranges of $\bar{n}^*,\bar{k}^*$ are definable in $\bbN$ even by a
  bounded formula. In fact, in our computations later we put $n^*_i =
  \beth(30i+30)$ (for $i>0$) and $k^*_i=\beth(30i+20)$, where $\beth(0)= 
  1$, $\beth(i+1) = 2^{\beth(i)}$.  

We also let $n_*(i) = n^*_i$. 
\item Let $\cY_\ell=\{\pi:\pi$ is a permutation of ${}^{n_*(\ell)}2\}$ and
  $\bT_n = \{\langle\pi_\ell:\ell < n\rangle:\pi_\ell\in \cY_\ell$ for
  $\ell< n\}$ and $\bT= \bigcup\{\bT_n:n < \omega\}$.

For $\varkappa\in \bT_n$ we keep the convention that $\varkappa = \langle 
\pi^\varkappa_\ell:\ell< n\rangle$ (unless otherwise stated).

\item For $\varkappa\in \bT$ let $<_\varkappa$ be the following partial
  order:  
\begin{enumerate}
\item[(a)]  ${\rm Dom}(<_\varkappa) = \bigcup\{{}^{n_*(i)}2:i
  <\lh(\varkappa)\}$; 
\item[(b)]  $\eta<_\varkappa\nu$ if and only if they are from ${\rm
    Dom}(<_\varkappa)$ and for some $i<j$ we have $\eta\in {}^{n_*(i)}2$,
  $\nu \in {}^{n_*(j)}2$ and $\pi^\varkappa_i(\eta)\triangleleft
  \pi^\varkappa_j(\nu)$.  
\end{enumerate}
Let $t_\varkappa = ({\rm Dom}(<_\varkappa),<_\varkappa)$ for
$\varkappa\in\bT$.   
\item Let $\bT_\omega$ be $\lim_\omega(\bT)$, i.e., 
\[\bT_\omega = \{\langle \pi_i:i < \omega\rangle:\pi_i\mbox{ is a
  permutation of }{}^{n_*(i)}2\mbox{ for }i<\omega\}\]
and for $\varkappa\in\bT_\omega$ let $\varkappa\rest n= \langle
\pi^\varkappa_i: i < n\rangle$.
 
We interpret $\varkappa\in\bT_\omega$ as the tree $t_\varkappa:= (\bigcup_{i
  < \omega} {}^{n_*(i)}2,<_\varkappa)$, where $<_\varkappa =\bigcup
\{<_{\varkappa \rest n}:n < \omega\}$, so $t = t_\varkappa$ is $({\rm Dom}(t),<_t)$.
\item Let $F$ be a one--to--one function from $\bigcup\{{}^{n_*(i)}2:i <
  \omega\}$ onto $\omega$, defined in $\bbN$ (i.e., the functions
  $n\mapsto\lh(F^{-1}(n))$ and $(n,i)\mapsto (F^{-1}(n))(i)$ are definable
  in $\bbN$ even by a bounded formula) such that $F$ maps each
  ${}^{n_*(i)}2$ onto an interval. Then clearly $F^{-1}$ is a one--to--one
  function from $\bbN$ onto $\bigcup\{{}^{n_*(i)}2:i < \omega\}$.  If
  $\bar{n}^*,\bar{k}^*$ are not definable in $\bbN$ then we mean definable
  in $(\bbN,\bar{n}^*,\bar{k}^*)$, considering $\bar{n}^*,\bar{k}^*$ as
  unary functions.
\item For $\varkappa\in\bT_\omega$ let $<^*_\varkappa$ be
  $\{(F(\eta),F(\nu)):\eta <_\varkappa \nu\}$ and $A_\varkappa =
  \{\pr(n_1,n_2):n_1 <^*_\varkappa n_2\}$ and let $t^*_\varkappa =
  (\omega,<^*_\varkappa)$; similarly $t^*_\varkappa$ for $\varkappa\in\bT$.  
\item For $\bS\subseteq \bT_\omega$ let $\cA_{\bS}=\{A_\varkappa: \varkappa
  \in \bS\}$ and let $\bA_{\bS}$ be the arithmetic closure of $\cA_{\bS}$
  recalling \ref{0z.2}(1). 
\end{enumerate}
\end{definition}

\begin{proposition}
\label{4d.5}
For $\varkappa\in\bT_\omega$, in $(\bbN,A_\varkappa)$ we can define
$<^*_\varkappa$ and 
\[\begin{array}{ll}
(\bbN,A_\varkappa)\models&\mbox{`` }<^*_\varkappa\mbox{ is a tree
with set of levels $\bbN$, set of elements $\bbN$ and }\\
&\quad \mbox{ each level finite (=bounded in $\bbN$, even an interval) ''.}
\end{array}\]
Of course, $t_\varkappa$ and $t_\varkappa^*=(\omega,<^*_\varkappa)$ are
isomorphic trees. Note that in $\bbN$ we can interpret the finite set theory
$\cH(\aleph_0)$.    
\end{proposition}

Our aim is to construct objects with the following properties.

\begin{definition}
\label{4d.7} 
\begin{enumerate}
\item We say $\bT^*_\omega$ is {\em strongly pcd\/} (perfect cone disjoint)
  whenever:\\ 
$\bT^*_\omega$ is a perfect subset of $\bT_\omega$ such that:
\begin{enumerate}
\item[$\boxtimes^{\text{st}}_{\bT^*_\omega}$]  if $n<\omega$ and
  $\varkappa_0, \varkappa_1, \ldots, \varkappa_n \in \bT^*_\omega$ with no
  repetitions and for $\ell=0,1$, $\eta_\ell$ is an $\omega$--branch of
  $t^*_{\varkappa_\ell}$ which is definable in $(\bbN,A_{\varkappa_\ell},
  A_{\varkappa_2}, \ldots, A_{\varkappa_n})$,\\ 
then $\eta_0,\eta_1$ belong to disjoint cones (in their respective trees)
which means that: 
\begin{enumerate}
\item[$(\boxdot)$] for some level $n$ the sets 
\[\{a:a \mbox{ is $<^*_{t_\ell}$--above the member of $\eta_\ell$ of level
}n\} \subseteq \bbN\]
for $\ell=0,1$ are disjoint.
\end{enumerate}
\end{enumerate}
\item We say $\bT^*_\omega$ is {\em weakly pcd\/} (perfect cone disjoint)
  whenever:\\
$\bT^*_\omega$ is a perfect subset of $\bT_\omega$ such that:
\begin{enumerate}
\item[$\boxtimes^{\text{wk}}_{\bT^*_\omega}$] for every $n$ and
  $\varphi(x,\bar{y}_\ell) \in \bbL(\tau_{\PA} + \{P_0,\ldots,P_n\})$ there
  is $i(*)$ such that {\em if}
\begin{itemize}
\item $i \in [i(*),\omega)$ and $\varkappa_{m,\ell} \in \bT^*_\omega$ for $m\le
 n$, $\ell = 0,1$, 
\item $\varkappa_{0,0}\neq\varkappa_{0,1}$ and 
\item $\varkappa_{m_1,\ell_1} \rest i = \varkappa_{m_2,\ell_2} \rest i$\quad
  if and only if\quad $m_1 = m_2$, and 
\item $P_0,\ldots,P_n$ are unary predicates, $\varphi=\varphi(x,\bar{y},
  P_0,\ldots,P_n) \in \bbL(\tau_{\PA}+\{P_0,\ldots,P_n\})$, and
  $\bar{b}_\ell \in {}^{\lh(\bar{y})}\bbN,\varphi(x,\bar{b}_\ell,
  A_{\varkappa_{0,\ell}}, \ldots, A_{\varkappa_{n,\ell}})$ define in
  $(\bbN,A_{\varkappa_{0,\ell}},\ldots,A_{\varkappa_{n,\ell}})$ a branch
  $B_\ell$ of $t^*_{\varkappa_{0,\ell}}$ for $\ell=0,1$ 
\end{itemize}
{\em then\/} the branches $B_0,B_1$ have disjoint cones (in their respective
trees). 
\end{enumerate}
\item Conditions $\otimes^{\text{wk}}_{\bT^*_\omega}$ and
  $\otimes^{\text{st}}_{\bT^*_\omega}$ are defined like
  $\boxtimes^{\text{wk}}_{\bT^*_\omega}$,
  $\boxtimes^{\text{st}}_{\bT^*_\omega}$ above replacing ``have disjoint
  cones'' (i.e., $(\boxdot)$) by ``have bounded intersection'', which means
  that 
\begin{enumerate}
\item[$(\odot)$] for some $a$ the sets $\{b\in\eta_0: b$ is of level $>a\}$
  and $\{b\in\eta_1: b$ is of level $>a\}$ are disjoint.
\end{enumerate}
Then we define {\em weakly pbd\/} and {\em strongly pbd\/} (where {\em
  pbd\/} stands for {\em perfect branch disjoint\/}) in the same manner as
pcd above, replacing   $\boxtimes^{\text{wk}}_{\bT^*_\omega}$,
$\boxtimes^{\text{st}}_{\bT^*_\omega}$ by
$\otimes^{\text{wk}}_{\bT^*_\omega}$ and
$\otimes^{\text{st}}_{\bT^*_\omega}$, respectively.
\item Omitting strongly/weakly means weakly. 
\end{enumerate}
\end{definition}

One may now ask if the existence of pcd/pbd (Definition \ref{4d.7}) can be
proved and if this concept helps us. We shall prove the existence of pbd in
Sections  2 and 3, specifically in \ref{3c.37}. The existence of pcd remains
an open question. Below we argue that objects of this kind are usefull to
prove Theorem \ref{4d.1}. 

\begin{theorem}
\label{4d.13} 
\begin{enumerate}
\item[(a)] If $\bT^*_\omega$ is a pcd, i.e., it is a perfect subset of $\bT_\omega$
satisfying $\boxtimes^{\rm wk}_{\bT^*_\omega}$ from Definition \ref{4d.7},
then $\cA=\cA_{\bT^*_\omega}$ (see Definition \ref{4d.3}(7)) is as required
in \ref{4d.1}.  
\item[(b)] Even if $\bT^*_\omega$ is a pbd then $\cA=\cA_{\bT^*_\omega}$ is as
  required in \ref{4d.1}.
\end{enumerate}
\end{theorem}

\begin{proof}
(a)\qquad We will deal with each part of Theorem \ref{4d.1}. First we give
details for part (3) of \ref{4d.1}.

For $\varkappa\in\bT^*_\omega$ recall 
\[A_\varkappa=\{\pr(F(\eta),F(\nu)):\eta <^*_\varkappa \nu\}\subseteq\bbN\]
and $\cA=\{A_\varkappa:\varkappa\in\bT^*_\omega\}\subseteq\cP(\omega)$.
Assume $\cA'\subseteq\cA$ is uncountable and let $T=T_{\cA'}=
\Th(\bbN_{\cA'})$ and $\tau_{\cA'}$ be its vocabulary. Then by \cite{Sh:73}
the theory $T$ has a model $M$ in which definable trees (we are interested
just in the case the set of levels being $M$ with the order $<^M$) have no
undefinable branches, so, in particular (and this is enough) 

if $\varkappa \in \cA$, then $(<^*_\varkappa)^M$ has no undefinable branch 

\noindent (i.e., as in \cite{Sh:73}, branches mean full branches,
``visiting'' every level).  Note that ``the $a$-th level of
$(M,(<^*_\varkappa)^M)$'' does not depend on $\varkappa$.  

Assume towards contradiction $M^+$ is an (elementary) end-extension of $M$
and let $b^*\in M^+\setminus M$.  Now consider any $A_\varkappa\in \cA$ so
$(<^*_\varkappa)^M$ is naturally definable in $M$ and 
\[\begin{array}{ll}
M \models&\mbox{`` for every element $a$ serving as level,}\\
&\langle\{c:b <_\varkappa c\}:b\mbox{ is of level $a$ in the tree $t_\varkappa$,
i.e. }(M,(<^*_\varkappa)^M)\rangle\\
&\mbox{ is a partition of $\{x:x$ is of $<^*_\varkappa$-level $>a\}$ to
  finitely many sets '',}
\end{array}\]
the finite is in the sense of $M$ of course.

As $M^+$ is an end-extension of $M$ recalling \ref{4d.3}(5) it follows that
the level of $b^*$ in $M^+$ is above $M$ and $b^*$ defines a branch of
$(M,(<^*_\varkappa)^M)$ which we call $\eta_\varkappa = \langle
b^\varkappa_a:a \in M\rangle$. That is $b^\varkappa_a$ is the unique member
of $M$ of level $a$ such that $M^+ \models$`` $b^\varkappa_a \le^*_\varkappa
b^*$ ''. 

By the choice of $M$ the branch $\eta_\varkappa$, i.e., $\{b^\varkappa_a:a
\in M\}$ is a definable subset of $M$, say by $\varphi_\varkappa(x,
\bar{d}_\varkappa)$ where $\varphi_\varkappa(x,\bar{y}_\varkappa) \in
\bbL(\tau_{\cA'})$ and $\bar{d}_\varkappa\in {}^{\lh(\bar{y}_\varkappa)}M$.
Now by the assumptions on $\cA,\cA',T$ there are $s_{\varkappa,1},\ldots,
s_{\varkappa,n_\varkappa}\in \bT^*_\omega\setminus\{\varkappa\}$ with no
repetitions, hence $A_{s_{\varkappa,n}} \in \cA'\setminus\{A_\varkappa\}$
for $n=1,\ldots, n_\varkappa$, and in $\varphi_\varkappa (x,
\bar{y}_\varkappa)$ only $A_{s_{\varkappa,1}},\ldots,A_{s_{\varkappa,
    n_\varkappa}}$ and $A_\varkappa$ appear (i.e., the predicates
$P_{s_{\varkappa,1}},\ldots, P_{s_{\varkappa, n_\varkappa}},P_\varkappa$
corresponding to them and $\tau_{\PA}$, of course). Let $s_{\varkappa,0}=
\varkappa$ and we write $\varphi'_\varkappa=\varphi'_\varkappa(x,
\bar{y}_\varkappa,\bar{P}_\varkappa)$, where $\bar{P}_\varkappa= \langle
P_{s_{\varkappa,\ell}}: \ell\leq n_\varkappa\rangle$ and
$\varphi'_\varkappa$ has non-logical symbols only from $\tau_{\PA}$ and so
$\varphi'_\varkappa=\varphi''_\varkappa(x,\bar{y}_\varkappa)\in
\bbL(\tau_\PA \cup \{P_\ell:\ell\leq n_\varkappa\})$, that is
$\varphi'_\kappa(x,\bar{y}_\varkappa)$ when we substitute $P_\ell$ for
$P_{s_{\varkappa,\ell}}$ for $\ell\leq n_\varkappa$. 

For $A_\varkappa\in\cA$ let 
\[m_\varkappa=\min\{m:s_{\varkappa,\ell} \rest m\mbox{ for } \ell=
0,\ldots, n_\varkappa\mbox{ are pairwise distinct }\}.\]  
Hence for some $\varphi_*(x,\bar{y}_*),n_*,m_*,\bar{s}_*$ the set 
\[\cA_2 = \{A_\varkappa\in\cA:\varphi'_\varkappa=\varphi_*,\
\bar{y}_\varkappa= \bar{y}_*,\ \mbox{ so }\ n_\varkappa = n_*,\ m_\varkappa
= m_*\mbox{ and } \langle s_{\varkappa,\ell} \rest
m_*:\ell=0,\ldots,n_*\rangle = \bar{s}_*\}\]
is uncountable. Let $i(*)\geq m_*$ be as guaranteed by
$\boxtimes^{\text{wk}}_{\bT^*_\omega}$, so for some uncountable $\cA_3
\subseteq \cA_2$ for some $\bar{s}_{**}$ we have that $\langle
s_{\varkappa,\ell} \rest i(*):\ell=1,\ldots,n_*\rangle = \bar{s}_{**}$
whenever $A_\varkappa \in \cA_3$. As $\cA$ is uncountable clearly for some
$A_{\varkappa_1} \ne  A_{\varkappa_2} \in \cA$ we have
$\{\varkappa_1,\varkappa_2\}$  is disjoint to $\{s_{\varkappa_\ell,m}:m = 1,
\ldots,n_{\varkappa_\ell}$ and $\ell=1,2\}$.

So by $\boxtimes^{\text{wk}}_{\bT^*_\omega}$ from Definition \ref{4d.7}
for some $a \in M$ we have  
\begin{enumerate}
\item[$(\boxdot)$] $M \models$`` $\{c:b^{\varkappa_1}_a <^*_{\varkappa_1}
  c\}\cap\{c:b^{\varkappa_2}_a <^*_{\varkappa_2} c\} = \emptyset$ ''.
\end{enumerate}
[Why?  Because $\bbN_{\cA'}\models$`` $(\forall \bar{y}_{\varkappa_1})
(\forall \bar{y}_{\varkappa_2})$ [if $\varphi_{\varkappa_\ell}(-,
\bar{y}_{\varkappa_\ell})$ define a branch of $t^*_{\varkappa_\ell}$ for
$\ell=1,2$, then there are $x_1,x_2$ such that $\varphi_{\varkappa_1}(x_1,
\bar{y}_{\varkappa_1})\wedge\varphi_{\varkappa_2}(x_2,
\bar{y}_{\varkappa_2}) \wedge \neg(\exists z)[x_1 \le_{t^*_{\varkappa_1}} z
\wedge x_2 \le_{t^*_{\varkappa_2}} z]$] ''.]

But in $M^+$ the elements $b^*$ belong to both, contradiction to $M \prec
M^+$.
\medskip

Now, parts (2), (3) of \ref{4d.1} follow and so does part (1).
\medskip

(4) See on this \cite{EnSh:936}. Alternatively, when is $\cB= \{A\subseteq
\bbN:A$ is definable in $\bbN_\cA\}$ Borel? As we can shrink $\bT^*_\omega$,
without loss of generality there is a function $g\in {}^\omega \omega$ such
that for every $f\in {}^\omega \omega$ definable in $\bbN_\cA$, we have $f
<_{J^{\text{bd}}_\omega} g$, i.e., $(\forall^\infty i))(f(i) < g(i))$.  This
suffices (in fact if we prove \ref{4d.7} using forcing notion $\bbQ_u$,
where each $\bbQ_u$ is ${}^\omega \omega$-bounding this will be true for
$\bT^*_\omega$ itself and we do this in \S3; moreover we have continuous
reading for every such $f$ (as a function of $(A_{\varkappa_0}, \ldots,
A_{\varkappa_{n-1}})$ for some $\varkappa_0, \ldots,\varkappa_{n-1} \in
\bT^*_\omega$).    
\medskip

\noindent (b)\qquad We repeat the proof of (a) above untill the choice of
$\{\varkappa_1,\varkappa_2\}$ (right before $(\boxdot)$), but we replace the
rest of the arguments for clause (3) of \ref{4d.1} by the following.

So by $\otimes^{\text{wk}}_{\bT^*_\omega}$ of Definition \ref{4d.7}(3),
for some $a_*\in M$ we have 
\begin{enumerate}
\item[$(\odot)$] $M \models$`` the sets $\{b^{\varkappa_1}_a: a_*<a\}$,
  $\{b^{\varkappa_2}_a: a_*<a\}$ are disjoint''.
\end{enumerate}
(Remember that all the trees we consider have the same levels.) But in $M^+$
the element $b^*$ belongs to both definable branches contrary to $M\prec
M^+$.  
\end{proof}

\begin{theorem}
\label{4d.17}
\begin{enumerate}
\item If $\bT^*_\omega$ is a strong pcd, i.e., it is a perfect subset of
  $\bT_\omega$ satisfying $\boxtimes^{\text{st}}_{\bT^*_\omega}$ from
  \ref{4d.7}, and $\cA \subseteq \{A_\varkappa:\varkappa \in \bT^*_\omega\}$
  is uncountable, then there is no weakly definably closed ultrafilter on
  $\arcl(\cA)$, see Definition \ref{0z.7}(5).
\item Above, we may replace ``pcd'' with ``pbd''.
\item Without loss of generality, $\arcl({\bT}^*_\omega)$ is a Borel set.
\end{enumerate}
\end{theorem}

\begin{proof}
(1) Assume towards contradiction that a pair $(\cA,D)$ forms a
counterexample.  Let $M =\bbN_\cA$ and let $M^+$ be an $\aleph_2$--saturated
elementary extension of $M$ and let $b^*\in M^+$ realizes the type
\[\begin{array}{ll}
p^* = \{\varphi(x,\bar{a}):&\varphi(x,\bar{y})\in \bbL(\tau_M), \bar{a}\in
{}^{\lh(\bar{y})} M \mbox{ and} \\ 
  &\{b\in M:M \models \varphi[b,\bar{a}]\} \mbox{ includes some member of }
  D\}. 
\end{array}\]
Clearly $p^*$ is a set of formulas over $M$, finitely satisfiable in
$M$ and even a complete type over $M$.

Now, for every $\varkappa$ such that $A_\varkappa \in \cA$ and $i<\omega$ we
consider a function $g_{\varkappa,i}$ definable in $M$ as follows:
\begin{enumerate}
\item[$(*)_1$]  $g_{\varkappa,i}(c)$ is:
\begin{enumerate}
\item[$(\alpha)$]  $b$ if $c$ is of $<^*_{\varkappa}$--level $\ge i$ in
  $(\bbN,<_\varkappa)$ and $b$ is of $<^*_\varkappa$--level $i$ and $b
  \le^*_\varkappa c$;
\item[$(\beta)$]  $c$ if $c$ is of $<^*_\varkappa$-level $<i$ in
  $(\bbN,<_\varkappa)$.  
\end{enumerate}
\end{enumerate}
Clearly $g_{\varkappa,i}$ is definable in $(\bbN,A_\varkappa)$, the range of
$g_{\varkappa,i}$ is finite, so $g_{\varkappa,i}\rest B_{\varkappa,i}$ is
constant for some $B_{\varkappa,i}\in \{g^{-1}_{\varkappa,i}\{x\}:x\in {\rm
  Rang}(g_{\varkappa,i})\}\cap D$.  As all co-finite subsets of $\bbN$
belong to $D$, also $B_{\varkappa,i}$ cannot be a singleton member of
level $\neq i$. Hence for some $b_{\varkappa,i}$ of level $i$ for
$<^*_\varkappa$ we have $B_{\varkappa,i}\subseteq \{c:b_{\varkappa,i}
\le^*_\varkappa c\}$.  Now moreover for some formula $\varphi_\varkappa
(x_0,x_1,x_2) \in\bbL(\tau_{\PA}+ P_\varkappa)$, for each $i\in\bbN$ the
formula $\varphi_\varkappa(x_0,x_1,i)$ defines $g_{\varkappa,i}(x_1) = x_1$. 
By the ``weakly definable closed'' (see Definition \ref{0z.7}(5)),
$\{b_{\varkappa,i}:i <\omega\}$ is definable in $\bbN_\cA$. 

Now we continue as in the proof of \ref{4d.13}.
\medskip

(2) Similarly.
\medskip

(3) As in \ref{4d.13} (for clause (4) of \ref{4d.1}).
\end{proof}

\section{The (iterated) creature forcing}
We continue the previous section, so we use notation as there, see
Definitons \ref{4d.3} and \ref{4d.7}. In particular, $n^*_0=0$,
$n_*(i)=n^*_i = \beth(30i+30)$ (for $i>0$) and $k^*_i=\beth(30i+20)$. We 
also set $\ell^*_i=\beth(30i+10)$.

\begin{definition}
\label{2q.1} 
For $i<\omega$ and a finite set $u$ of ordinals we define:
\begin{enumerate}
\item[(A)]  $\OB^u_i$ is the set of all triples $(f,g,e)$ such that ($\Per(A)$
  stands for the set of permutations of $A$): 
\begin{enumerate}
\item[(a)]  $f,g\in {}^u(\Per({}^{n_*(i)}2))$;
\item[(b)]  if $i-1=j\ge 0$ and $\alpha\in u$, then $(f(\alpha)(\rho)) \rest
  n^*_j = (g(\alpha)(\rho))\rest n^*_j$ for all $\rho \in {}^{n_*(i)}2$,
\item[(c)]  $e$ is a function with domain $u$ such that for each $\alpha\in
  u$ 
\[e(\alpha):\Per({}^{n_*(i-1)}2)\longrightarrow \Per({}^{n_*(i)}2) \times
 \Per({}^{n_*(i)}2).\]
\end{enumerate}
Above, we stipulate $n_*(i-1)=0$ if $i=0$. Also, let us note that some
triples will never be used, only $\bigcup\{\suc(x):x\in\OB^u_i\}$ and we
should iterate. 
\item[(B)] For $x\in\OB^u_i$ we let $x = (f_x,g_x,e_x)$ and $i=\bi(x)$ and 
  $u=\supp(x)$.  
\item[(C)] For $x\in\OB^u_i$ we set 
\[\begin{array}{lr}
\suc(x)=\big\{y\in\OB^u_{i+1}:& \big(\forall\rho \in {}^{n_*(i+1)}2 \big)
\big(\forall\alpha\in u\big) \big(g_x(\alpha)(\rho\rest
n^*_i)=(f_y(\alpha)(\rho)) \rest n^*_i \big)\mbox{ and }\ \ \ \\
& \big(\forall\alpha\in u\big)\big(e_y(\alpha)(g_x(\alpha))=
(f_y(\alpha),g_y(\alpha))\big)\big\}.  \end{array}\]   
\item[(D)]  For $j \le \omega$ let 
\[\bS_{u,j} = \big\{\langle x_\ell:\ell < j\rangle: (\ell< j \Rightarrow
x_\ell\in\OB^u_\ell)\ \mbox{ and }\ (\ell+1 < j \Rightarrow x_{\ell+1} \in 
\suc(x_\ell))\big\}.\] 
\item[(E)]  $\bS_u = \bigcup\{\bS_{u,\ell}:\ell<\omega\}$; we consider it a
  tree, ordered by $\triangleleft$. 
\item[(F)] For $x\in\OB^u_i$ and $w\subseteq u$ let $x\rest w= (f_x\rest w,
  g_x \rest w, e_x\rest w)$. 
\item[(G)] For $i\le \omega$, $w \subseteq u$ and $\bar{x}=\langle x_j:j<
  i\rangle\in\bS_{u,i}$ let $\bar{x} \upharpoonleft w = \langle x_j \rest
  w:j<i\rangle$ and for $\alpha \in u$ let $\varkappa^\alpha_{\bar{x}} =
  \langle f_{x_j}(\alpha):j<i\rangle$. 
\item[(H)] For $\bar{x}\in\bS_{u,\ell}$, $\ell\le\omega$, and $\alpha\in u$
  let $t_{\bar{x},\alpha}=t^\alpha_{\bar{x}}$ be the tree with
  $\lh(\bar{x})$ levels, with the $i$-th level being ${}^{n_*(i)}2$ for
  $i<\lh(\bar{x})$ and the order $<_{t_{\bar{x},\alpha}}$ defined by 

$\eta <_{t_{\bar x,\alpha}} \nu$\qquad if and only if 

for some $i<j<\lh(\bar{x})$ we have $\eta\in {}^{n_*(i)}2$, $\nu\in
{}^{n_*(j)}2$ and $f_{x_i}(\alpha)(\eta) \triangleleft
f_{x_j}(\alpha)(\nu)$.  
\end{enumerate} 
\end{definition}

Since we are interested in getting ``bounded branch intersections'' we will
need the following observation (part (5) is crucial in proving cone
disjointness in some situation later).

\begin{proposition}
\label{2q.3}
Assume $\bar{x}\in\bS_u$ and $\alpha \in u$.
\begin{enumerate}
\item If $\rho\in {}^{n_*(j)}2$ and $j<\lh(\bar{x})$, then $\langle
  g_{x_i}(\alpha)(\rho \rest n_*(i)):i \le j\rangle$ is
  $\triangleleft$--increasing noting $g_{x_i}(\alpha)(\rho\rest n_*(i)) \in
  {}^{n_*(i)}2$. 
\item $\varkappa^\alpha_{\bar{x}}\in\bT_{\lh(\bar{x})}$ and
  $t_{\varkappa^\alpha_{\bar{x}}} = t^\alpha_{\bar{x}}$, on
  $t_{\varkappa^\alpha_{\bar{x}}}$ see \ref{4d.3}(3). 
\item If $i<j<\lh(\bar{x})$ and $\nu\in {}^{n_*(j)}2$, then
  $(f_{x_j}(\alpha)(\nu)) \rest n^*_i$ depends just on $\bar{x}\rest (i+1)$,
  actually just on $g_{x_i}$, i.e., it is equal to $g_{x_i}(\alpha)(\nu\rest
  n^*_i)$.  
\item The sequence $\langle g_{x_j}(\alpha),f_{x_j}(\alpha):j<\lh (\bar{x})
  \rangle$ is fully determined by $\langle e_{x_j}(\alpha):j<\lh(\bar{x})
  \rangle$.  
\item Assume $\alpha_1 \ne \alpha_2$ are from $u$ and $i<\lh(\bar{x})$ and
  $\eta_1,\eta_2 \in {}^{\eta_*(i)}2$ but 
\[(g_{x_i}(\alpha_1))^{-1} \circ f_{x_i}(\alpha_1))(\eta_1) \ne
   ((g_{x_i}(\alpha_2))^{-1} \circ f_{x_i}(\alpha_2))(\eta_2).\]
Then the sets $\{\rho:\eta_1 <_{t_{\bar{x},\alpha_1}} \rho\}$ and 
$\{\rho:\eta_2 <_{t_{\bar{x},\alpha_2}} \rho\}$ are disjoint. 
\end{enumerate}
\end{proposition}

\begin{proof}  
(1), (2), (3) and (4) can be shown by straightforward induction on $j$.

(5) Assume towards contradiction that
\begin{enumerate}
\item[$(*)_1$]  $\eta_1<_{t_{\bar{x},\alpha_1}}\rho$ and
  $\eta_2<_{t_{\bar{x}, \alpha_2}} \rho$.
\end{enumerate}
So $\rho\in t_{\bar{x},\alpha_2}$ and hence $\rho \in {}^{n_*(j)}2$ for some
$j<\lh(\bar{x})$. Since $\eta_1<_{t_{\bar{x},\alpha_1}}\rho$, necessarily
$i<j< \lh(\bar{x})$ and by the definition of $<_{t_{\bar{x},\alpha_1}}$ and
$<_{t_{\bar{x},\alpha_2}}$: 
\begin{enumerate}
\item[$(*)_2$]  $f_{x_i}(\alpha_1)(\eta_1)\triangleleft
  f_{x_j}(\alpha_1)(\rho)$ and $f_{x_i}(\alpha_2)(\eta_2) \triangleleft
  f_{x_j}(\alpha_2)(\rho)$.
\end{enumerate} 
This means that
\begin{enumerate}
\item[$(*)_3$]  $f_{x_i}(\alpha_1)(\eta_1) = (f_{x_j}(\alpha_1)(\rho)) \rest
  n^*_i$ and $f_{x_j}(\alpha_2)(\eta_2)= (f_{x_j}(\alpha_2)(\rho)) \rest
  n^*_i$.
\end{enumerate}
Consequently, by part (3), letting $\rho'=\rho\rest n^*_i$:
\begin{enumerate}
\item[$(*)_4$]  $f_{x_i}(\alpha_1)(\eta_1)=g_{x_i}(\alpha_1)(\rho')$ and
  $f_{x_i}(\alpha_2)(\eta_2)= g_{x_i}(\alpha_2)(\rho')$,
\end{enumerate}
and therefore
\begin{enumerate}
\item[$(*)_5$]  $((g_{x_i}(\alpha_1))^{-1} \circ f_{x_i}(\alpha_1))(\eta_1)
  = \rho' = ((g_{x_i}(\alpha_2))^{-1} \circ f_{x_i}(\alpha_2))(\eta_2)$,
\end{enumerate}
contradicting our assumptions. 
\end{proof}

Below we may replace the role of $D^u_i$ by $\{\langle (f_{x_j}(\alpha), 
g_{x_j}(\alpha)):j<i\rangle: \bar{x}\in\bS_{u,i}\}$.

\begin{definition}
\label{2q.8} 
For a finite set $u \subseteq {\rm Ord}$ and an integer $i < \omega$ we let  
\begin{enumerate}
\item[(I)]
\begin{enumerate}  
\item[$(\alpha)$] $D^u_i = \{(\alpha,g):\alpha \in u$ and
  $g\in\Per({}^{n_*(i-1)}2)$ if $i>0$, $g\in\Per({}^0 2)$ if $i=0\}$;

if $\bar x \in\bS_{u,i}$ and $\alpha \in u$, then stipulate
$g_{x_{-1}}(\alpha)$ is the unique $g\in\Per({}^0 2)$.
\item[$(\beta)$] $\pos^u_i$ is the set of all functions $h$ with domain
  $D^u_i$ such that $h(\alpha,g)$ is a pair $(h_1(\alpha,g),h_2(\alpha,g))$ 
  satisfying 
\begin{itemize}
\item $h_1(\alpha,g),h_2(\alpha,g) \in \Per({}^{n_*(i)}2)$, and
\item $(h_\ell(\alpha,g)(\rho))\rest n_*(i-1) = g(\rho\rest n_*(i-1))$
  for $\ell \in \{1,2\}$, $i > 0$ and $\rho \in {}^{n_*(i)}2$.
\end{itemize}
Also, for $h\in\pos^u_i$ and $w\subseteq u$ we let $h\upharpoonleft w=
h\rest D^w_i$.
\item[$(\gamma)$] $\wpos^u_i$ is the family of all functions
  $\cF:\pos^u_i\longrightarrow [0,1]$ which are not constantly zero, and 
\[\vpos^u_i=\Big\{\cF\in\wpos^u_i:{\rm range}(\cF)\subseteq
\big\{\frac{m}{2^{n_*(i)}}:m=0,1,\ldots 2^{n_*(i)}\big\}\Big\}.\]
If above we allow the constantly zero function instead of $\wpos^u_i,
\vpos^u_i$ we get $\ypos^u_i,\xpos^u_i$, respectively. A set
$A\subseteq\pos^u_i$ will be identified with its characteristic function
$\chi_A\in\vpos^u_i$.  
\item[$(\delta)$] For $\cF\in\wpos^u_i$ we let
\[\set(\cF)=\{h\in\pos^u_i:\cF(h)>0\}\quad \mbox{ and }\quad \|\cF\|=
\sum\{\cF(h): h\in\pos^u_i\}.\]
If $|\pos^u_i|\geq \|\cF\|\cdot (k^*_i)^{3^{k^*_i}-1}$, then we put
$\nor^0_i(\cF)=0$; otherwise we let 
\[\nor^0_i(\cF)=k^*_i-\log_3\Big(\log_{k^*_i}\Big(\frac{k^*_i\cdot
  |\pos^u_i|}{\|\cF\|}\Big)\Big).\]
\item[$(\varepsilon)$] For $\cF_1,\cF_2\in\wpos^u_i$ we let 
\begin{itemize}
\item $\cF_1\leq \cF_2$ if and only if $(\forall h\in\pos^u_i)( \cF_1(h)\leq
  \cF_2(h))$;
\item $(\cF_1+\cF_2)(h)=\cF_1(h)+\cF_2(h)$ and $(\cF_1\cdot
  \cF_2)(h)=\cF_1(h) \cdot \cF_2(h)$ for $h\in\pos^u_i$;
\item $[\cF_1]$ is the function from $\pos^u_i$ to $\{\frac{m}{2^{n_*(i)}}:
  m=0,1, \ldots, 2^{n_*(i)}\}$ given by 
\[[\cF_1](h)=\lfloor \cF_1(h)\cdot 2^{n_*(i)}\rfloor\cdot 2^{-n_*(i)}\qquad
\mbox{ for }h\in \pos^u_i.\]
\end{itemize}
\item[$(\zeta)$] For $\bar{x}\in\bS_{u,i}$ and $h\in\pos^u_i$ we let
  $\suc_{\bar{x}}(h)$ be $\bar{x}\conc\langle y \rangle$ where $y\in\OB^u_i$ 
  is defined by: 
\begin{itemize}
\item $(f_y(\alpha),g_y(\alpha))=h(\alpha,g_{x_{i-1}}(\alpha))$ for
$\alpha\in u$,
\item $e_y(\alpha)(\pi)=h(\alpha,\pi)$ for $\alpha\in u$ and $\pi\in 
\Per({}^{n_*(i-1)}2)$.
\end{itemize}
\end{enumerate}
\item[(J)] 
\begin{enumerate}
\item[$(\alpha)$] $\underline{\CR}^u_i$ is the set of all pairs $\gc=(\cF,m)
  = (\cF_{\gc}, m_{\gc})$ such that $m$ is a non-negative real and $\cF\in 
  \wpos^u_i$ and $\nor^0_i(\cF) \ge m$. We also let $\CR^u_i=\{\gc\in
  \underline{\CR}^u_i:\cF_\gc\in \vpos^u_i\}$.
\item[$(\beta)$] For $\gc\in\underline{\CR}^u_i$, we let $\nor^1_i(\gc)= 
  (\nor^0_i(\cF_{\gc}) -m_{\gc})$ and $\nor^2_i(\gc)=\log_{\ell^*_i}( 
  \nor^1_i(\gc))$ if non-negative and well defined, and it is
  zero otherwise. (Remember, $\ell^*_i=\beth(30i+10)$.) We will write
  $\nor_i(\gc)=\nor^2_i(\gc)$.  
\item[$(\gamma)$] For $\gc\in\underline{\CR}^u_i$ let
  $\underline{\Sigma}(\gc)$ be the set of all $\gd\in \CR^u_i$ such that
  $\cF_{\gd}\leq \cF_{\gc}$ and $m_{\gd}\ge m_{\gc}$. For $\gc\in
  \CR^u_i$ we let $\Sigma(\gc)=\underline{\Sigma}(\gc)\cap\CR^u_i$.  
\end{enumerate}
\item[(K)]  $\bbQ_u=(\bbQ_u,\leq_{\bbQ_u})$ is defined by 
\begin{enumerate}
\item[$(\alpha)$] conditions in $\bbQ_u$ are pairs $p=(\bar{x},\bar{\gc}) =
  (\bar{x}_p,\bar{\gc}_p)$ such that 
\begin{enumerate}
\item[(a)] $\bar{x}\in\bS_{u,i}$ for some $i=\bi(p)<\omega$, so $\bar{x}_p =
  \langle x_{p,j}:j<\bi(p)\rangle$, 
\item[(b)] $\bar{\gc}=\langle {\gc}_j:j\in [\bi(p),\omega)\rangle$, so
  ${\gc}_j = {\gc}^p_j$, and ${\gc}_j\in\CR^u_j$,
\item[(c)] the sequence $\langle \nor_j({\gc}_j):j\in [\bi(p),\omega)
  \rangle$ diverges to $\infty$; 
\end{enumerate}
\item[$(\beta)$]  $p \le_{\bbQ_u} q$\quad if and only if\quad (both are
from $\bbQ_u$ and)
\begin{enumerate}
\item[(a)] $\bar{x}_p \trianglelefteq \bar{x}_q$, and 
\item[(b)] if $\bi(p)\le j<\bi(q)$, then for some
  $h\in\set(\cF_{{\gc}^p_j})$ we have $\bar{x}_q\rest (j+1) =
  \suc_{\bar{x}_q \rest j}(h)$ (see clause (I)$(\zeta)$ above), 
\item[(c)] if $i\in [\bi(q),\omega)$, then ${\gc}^q_i\in\Sigma({\gc}^p_i)$.
\end{enumerate}
\end{enumerate}
$\underline{\bbQ}_u=(\underline{\bbQ}_u,\leq_{\underline{\bbQ}_u})$ is defined
similarly, replacing $\CR^u_j$, $\Sigma$ by $\underline{\CR}^u_j$,
$\underline{\Sigma}$, respectively.  
\item[(L)] If $u_1,u_2\subseteq{\rm Ord}$ are finite, $|u_1|=|u_2|$ and
  $h:u_1 \longrightarrow u_2$ is the order preserving bijection, then
  $\hat{h}$ is the isomorphism from $\bbQ_{u_1}$ onto $\bbQ_{u_2}$ induced
  by $h$ in a natural way.
\end{enumerate}
\end{definition}

\begin{proposition}
\label{2q.23}
Let $u\subseteq{\rm Ord}$ be a finite non-empty set, $i \in (1,\omega)$ and
$|u|\le n_*(i-1)$. Then
\begin{enumerate}
\item[(a)] $|\pos^u_{i-1}|<\beth(30i+3)$, $|\vpos^u_{i-1}|<\beth(30i+4)$, 
  $\nor^0_i(\pos^u_i)=k^*_i$ and $\nor_i(\gc_{u,i}^{\max})=
  \beth(30i+19)/\beth(30i+9)$ and $\CR^u_i=\Sigma(\gc_{u,i}^{\max})$,
  where $\gc^{\max}_{u,i}=(\pos^u_i,0)$. 
\item[(b)]  $|\bS_{u,i}|<\ell^*_i$ and if $\bar{x}\in\bS_{u,i}$ and
  $h\in\pos^u_i$, then $\suc_{\bar{x}}(h)\in \bS_{u,i+1}$. 
\item[(c)] If $\cF_1\leq \cF_2$ are from $\wpos^u_i$, then $0\leq
  \nor^0_i(\cF_1) \leq \nor^0_i(\cF_2)$.
\item[(d)] If $\gc\in\underline{\CR}^u_i$ and $\nor^1_i(\gc)\geq 1$, then
  $\gc$ has $k^*_i$--bigness with respect to $\nor^1_i$, which means that:\\ 
if $\cF_\gc=\sum\{\cY_k:k<k^*_i\}$ then $\nor^1_i(\gc)\le \max\{
\nor^1_i(\cY_m, m_\gc)+1:k<k^*_i\}$;\\
moreover, if $\cF'\leq\cF_{\gc}$, $\|\cF'\| \ge \|\cF_\gc\|/k^*_i$ then 
$\nor^0_i(\cF') \ge \nor^0_i(\cF_\gc)-1$. 
\item[(e)]  Both $\CR^u_i$ and $\underline{\CR}^u_i$ have halving with
  respect to $\nor^1_i$, that is  
\begin{enumerate}
\item[$(\alpha)$] if $\gc=(\cF_{\gc},m_{\gc})$, $m_1=(\nor^0_i(\cF_\gc) +
  m_\gc)/2$, $\gd=(\cF_\gc,m_1)$, then $\nor^1_i(\gd)\ge\nor^1_i(\gc)/2$,
  and  
\item[$(\beta)$] if $\gd'\in\Sigma(\gd)$ is such that $\nor^1_i(\gd') \ge
  1$, then $\gd'' := (\cF_{{\gd}'},m_\gc)$ satisfies 
\[\gd''\in \Sigma(\gc),\qquad \nor^1_i(\gd'') \ge\nor^1_i(\gc)/2\quad \mbox{
  and }\quad\cF_{\gd''} = \cF_{\gd'}.\] 
\end{enumerate}
\end{enumerate}
\end{proposition}

\begin{proof}
{\em Clause (a)\/}:\quad  Clearly by the definition $\gc^{\max}_{u,i}= 
(\pos^u_i,0)\in\CR^u_i=\Sigma(\gc^{\max}_{u,i})$ and  
\[\nor^0_i(\pos^u_i) =k^*_i-\log_3\big(\log_{k^*_i}(k^*_i)\big)=k^*_i,\]  
so $\nor_i^1(\gc^{\max}_{u,i}) = k^*_i-0 = k^*_i$ and
$\nor_i(\gc^{\max}_{u,i})=\log_{\ell^*_i}(k^*_i)=\log_{\beth(30i+10)}
\big(\beth(30i+20)\big)=\log_2\big(\beth(30i+20)\big)/ \log_2\big(
\beth(30i+10)\big)=\beth(30i+19)/\beth(30i+9)$. Now, for every $j>0$,
letting $A_j=\Per({}^{n_*(j)}2)\times \Per({}^{n_*(j)}2)$ and recalling
\ref{2q.8}(I)($\alpha$), we have 
\[|D^u_j|\le (2^{n_*(j-1)}!) \times |u| \le 2^{(2^{n_*(j-1)})^2} \times
|u|\quad\mbox{ and }\quad |A_j| \le (2^{n_*(j)}!)^2\le 2^{2^{2n_*(j)+1}}\le
2^{2^{3n_*(j)}}.\]  
Since $|u|\leq  n_*(i-1)$, we get $|D^u_j|\leq 2^{2^{2n_*(j-1)}}\times
n_*(i-1)$. Since $2^{2^{2n_*(i-2)}}\leq n_*(i-1)$, $n_*(i-1)^2\leq
2^{n_*(i-1)}$ and $4n_*(i-1)+1\leq 2^{n_*(i-1)}$, we conclude now that   
\[|\pos^u_{i-1}|\leq |A_{i-1}|^{|D^u_{i-1}|} \le
(2^{2^{3n_*(i-1)}})^{|D^u_{i-1}|} \leq 2^{2^{3n_*(i-1)}\times
2^{2^{2n_*(i-2)}}\times n_*(i-1)}\leq 2^{2^{4n_*(i-1)}}<\beth(30i+3)\]
and 
\[|\vpos^u_{i-1}|=(2^{n_*(i-1)}+1)^{|\pos^u_{i-1}|}<2^{(n_*(i-1)+1)\times
  2^{2^{4n_*(i-1)}}} < 2^{2^{2^{4n_*(i-1)+1}}}<\beth(30i+4).\]  
\smallskip

\noindent {\em Clause (b)\/}:\quad Let $B_j$ be the set of all functions
from $\Per({}^{n_*(j-1)}2)$ to $\Per({}^{n_*(j)}2)\times
\Per({}^{n_*(j)}2)$. Then we have
\[|B_j|=\Big(2^{n_*(j)}!\Big)^{2\cdot (2^{n_*(j-1)}!)}\leq
2^{2^{2n_*(j)}\cdot 2\cdot (2^{n_*(j-1)}!)} \leq 2^{2^{4n_*(j)}}\]
and hence for $j<i$:
\[\begin{array}{r}
\displaystyle 
|\OB^u_j|\leq |{}^u\Per({}^{n_*(j)}2)|\cdot |{}^u\Per({}^{n_*(j)}2)|\cdot 
|{}^u B_j|\leq \big(2^{n_*(j)}!)^{2|u|}\cdot 2^{2^{4n_*(j)}\cdot |u|}\leq\\
\ \\
\displaystyle 
2^{2^{2n_*(j)+1}\cdot |u|+2^{4n_*(j)}\cdot |u|}\leq 2^{2^{7n_*(j)}\cdot
  n_*(i-1)}\leq 2^{2^{8n_*(i-1)}}.
\end{array}\] 
Therefore,
\[|\bS_{u,i}|\le\prod\limits_{j<i}|\OB^u_j| \le
(2^{2^{8n_*(i-1)}})^i<2^{2^{9n_*(i-1)}}<\ell^*_i.\]
\smallskip

\noindent {\em Clause (d)\/}:\quad  Assume $\gc\in\underline{\CR}^u_i$ and
$\cF_\gc =\sum\{\cY_k:k < k^*_i\}$, hence $\|\cF_\gc\|=\sum\{\|\cY_k\|:k<
k^*_i\}$. Let $k(*)<k^*_i$ be such that $\|\cY_{k(*)}\|$ is maximal. Plainly 
$\|\cF_\gc\|\le k^*_i\times \|\cY_{k(*)}\|$ and therefore it 
suffices to prove the ``moreover'' part. So assume $\cY\leq \cF_\gc$,
$\|\cF_\gc\|\le k^*_i \times \|\cY\|$. Then  
\[\begin{array}{l}
\displaystyle 
  \nor^0_i(\cY)=k^*_i-\log_3\Big(\log_{k^*_i}\Big(\frac{k^*_i\cdot
  |\pos^u_i|}{\|\cY\|}\Big)\Big) \geq k^*_i-\log_3 \Big(\log_{k^*_i}
  \Big(\frac{k^*_i\cdot |\pos^u_i|}{\|\cF_\gc\|}\cdot k^*_i\Big)\Big)\geq\\
\ \\
\displaystyle 
k^*_i-\log_3\Big(3\log_{k^*_i}\Big(\frac{k^*_i\cdot
  |\pos^u_i|}{\|\cF_\gc\|}\Big)\Big)=\nor^0_i(\cF_\gc)-1,
\end{array}\]
so we are done.
\medskip

\noindent {\em Clauses (c) and (e)\/}:\quad  Obvious. 
\end{proof}

\begin{observation}
\label{2q.9}
\begin{enumerate}
\item $\bbQ_u$, $\underline{\bbQ}_u$ are non-trivial partial orders.
\item $\bbQ_u$ is a dense subset of $\underline{\bbQ}_u$.
\end{enumerate}
\end{observation}

\begin{proof}
(1)\quad Should be clear.

\noindent (2)\quad For $\gc\in\underline{\CR}^u_i$ such that
$\nor^1_i(\gc)>1$ we set $[\gc]=([\cF_\gc],m_\gc)$ (see
\ref{2q.8}(I)($\varepsilon$)). Note that 
$\frac{\|[\cF_\gc]\|}{|\pos^u_i|} \geq\frac{\|\cF_\gc\|}{|\pos^u_i|}- 
\frac{1}{2^{n_*(i)}}$ and hence (as $(k^*_i)^{3^{k^*_i}}<2^{n_*(i)}$ and
$\frac{\|\cF_\gc\|}{|\pos^u_i|}> (k^*_i)^{1-3^{k^*_i}}$) we have
$\frac{\|[\cF_\gc]\|}{|\pos^u_i|}\geq \Big(\frac{\|\cF_\gc\|}{|
  \pos^u_i|}\Big)^3\cdot \frac{1}{k^2}$ and hence easily
$\nor^0_i([\cF_\gc])\geq \nor^0_i(\cF_\gc)-1$. Consequently, $[\gc]\in  
\CR^u_i$ and $\nor^1_i([\gc])\geq \nor^1_i(\gc)-1$. 

Now suppose that $p\in \underline{\bbQ}_u$. We may assume that
$\nor_i(\gc^p_i)>1$ for all $i\geq \bi(p)$. Put $\bi(q)=\bi(p)$,
$\gc^q_i=[\gc^p_i]$ for $i\geq \bi(q)$ and $\bar{x}_q=\bar{x}_p$. Then
$q=(\bar{x}_q,\langle \gc^q_i:i\geq \bi(q)\rangle)\in\bbQ_u$ is a condition 
stronger than $p$.
\end{proof}

\begin{definition}
\label{2q.34}
Let $u \subseteq \text{ Ord}$ be a finite non-empty set.
\begin{enumerate}
\item Let $\name{\bar{x}}$ and $\name{\varkappa}_\alpha,\name{t}_\alpha$ for
  $\alpha\in u$ be the following $\bbQ_u$-names: 
\begin{enumerate}
\item[(a)]  $\name{\bar{x}}=\name{\bar{x}}_u=\bigcup\{ \bar{x}_p: p\in 
  \name{\bG}_{\bbQ_u}\}$ and $\name{\varkappa}_\alpha=\langle
  \name{\pi}_{\alpha,i}:i<\omega\rangle$, where 
\[\name{\pi}_{\alpha,i}[\name{\bG}_{\bbQ_u}]=\pi\quad \mbox{ if and only if
\quad  for some }p\in\name\bG\mbox{ we have }\lh(\bar{x}_p)>i\mbox{ and }
f_{x_{p,i}}(\alpha)=\pi.\]  
\item[(b)] $\name{t}_\alpha=t^*_{\name{\varkappa}_\alpha}$, i.e., it is a
  tree (see \ref{4d.3}(4)).
\end{enumerate}
\item For $p\in\bbQ_u$ let $\pos(p)=\{\bar{x}_q:p \le_{\bbQ_u} q\}$ and for
  $\bar{x}\in\pos(p)$ let $p^{[\bar{x}]}=(\bar{x}, \langle \gc^p_i:i \in
  [\lh(\bar{x}),\omega)\rangle)$. 
\end{enumerate}
\end{definition}

\begin{observation}
\label{2q.37}
Let $u\subseteq\text{\rm Ord}$ be a finite non-empty set, $\alpha \in
u$. Then:  
\begin{enumerate}
\item $\Vdash_{\bbQ_u}$`` $\name{\bar{x}}\in \bS_{u,\omega}$''. 
\item We can reconstruct $\name{\bG}_{\bbQ_u}$ from $\name{\bar{x}}$. As a
  matter of fact, $\langle e_{\name{\bar{x}}_i}:i<\omega\rangle$ determines
  $\langle f_{\name{\bar{x}}_i},g_{\name{\bar{x}}_i}:i<\omega\rangle$ (and
  also $\name{\bG}_{\bbQ_u}$). 
\item  $\name{\varkappa}_\alpha=\bigcup\{\varkappa_{\bar{x}}^\alpha: \bar{x}
  = \bar{x}_p$ and  $p\in\name{\bG}_{\bbQ_u}\}$.
\item $\Vdash_{\bbQ_u}$`` $\name{\varkappa}_\alpha\in\bT_\omega$ ''.
\item If $h:u\longrightarrow {\rm Ord}$ is one-to-one, then $\hat{h}$ (see
  \ref{2q.8}(L)) maps $\name{\bar{x}}_u$ to $\name{\bar{x}}_{h[u]}$,
  $(\name{\bar{x}}_u)_i$ to $(\name{\bar{x}}_{h[u]})_i$, etc. 
\end{enumerate}
\end{observation}

\begin{observation}
\label{2q.40}
\begin{enumerate}
\item $p^{[\bar{x}]}\in\bbQ_u$ and $p\le_{\bbQ_u} p^{[\bar{x}]}$ for every
  $\bar{x} \in \pos(p)$. 
\item If $p\in\bbQ_u$ and $i\in [\lh(\bar{x}_p),\omega)$, then the set
$\cI_{p,i}:=\{p^{[\bar{x}]}:\bar{x}\in\pos(p)\cap \bS_{u,i}\}$
is predense above $p$ in $\bbQ_u$.
\end{enumerate}
\end{observation}

\begin{proposition}
\label{2q.43}
$\bbQ_u$ is a proper ${}^\omega\omega$--bounding forcing notion with rapid
continuous reading of names, i.e., if $p\in\bbQ_u$ and $p\Vdash$``
$\name{h}$ is a function from $\omega$ to $\bV$ '', then for some
$q\in\bbQ_u$ we have: 
\begin{enumerate}
\item[(a)] $p \le q$ and $\bi(p) = \bi(q)$, 
\item[(b)] for every $i<\omega$ the set $\{y:q\nVdash_{\bbQ_u}$``
  $\name{h}(i) \ne y$ ''$\}$ is finite, moreover, for some $j\in
  [\lh(\bar{x}_q),\omega)$, for each $\bar{x}\in\pos(q)\cap\bS_{u,j}$ the
  condition $q^{[\bar{x}]}$ forces a value to $\name{h}(i)$, 
\item[(c)] if $p\Vdash_{\bbQ_u}$`` $(\forall i<\omega)(\name{h}(i)< k^*_i)$
  '', then:  
\begin{enumerate}
\item[$(\circledast)$] if $\bar{x}\in\pos(q)$ has length $i>\bi(q)$, then 
  $q^{[\bar{x}]}$ forces a value to $\name{h}(i)$. 
\end{enumerate}
\end{enumerate}
\end{proposition}

\begin{proof}
It is a consequence of \cite{RoSh:470}, so in the proof below we
will follow definitions and notation as there. First note that we may
assume $|u|<\bi(p)$ (as otherwise we fix $i>|u|$ and we carry out the
construction successively for all $\bar{x}\in\pos(p)$ of length $i$). 

For $i<\bi(p)$ let $\bH(i)=\{x_{p,i}\}$ and for $i\geq \bi(p)$ let
$\bH(i)=\pos^u_i$. Let $K^*$ consists of all creatures
$t=(\nor[t],\val[t],\dis[t])$ such that
\begin{itemize}
\item for some $i\geq \bi(p)$ and $\gc\in\CR^u_i$ we have $\dis[t]=(\gc,i)$
  and $\nor[t]=\nor^1_i(\gc)$, and
\item $\val[t]=\{(\bar{w},\bar{w}\conc\langle h\rangle):\bar{w}\in
  \prod\limits_{j<i} \bH(j)\ \&\ h\in \set(\cF_{\gc})\}$.
\end{itemize} 
(Note the use of $\nor^1_i$ and not $\nor^2_i$ above.) For $t\in K^*$ with 
$\dis[t]=(\gc,i)$ we let 
\[\Sigma^*(t)=\{s\in K: \dis[s]=(\gd,i)\ \&\ \gd\in\Sigma(\gc)\}.\]
Then $(K^*,\Sigma^*)$ is a local finitary big creating pair (for $\bH$) with
the Halving Property (remember \ref{2q.23}(d,e)). Now define
$f:\bbN\times\bbN\longrightarrow \bbN$ by $f(j,i)=(\ell^*_i)^{j+1}$. Let
$p^*\in\bbQ^*_f(K^*,\Sigma^*)$ be a condition such that $w^{p^*}=\bar{x}_p$
and $\dis[t^{p^*}_i]=(\gc^p_{i+\bi(p)},i+\bi(p))$ for $i<\omega$. Note that
$\bbQ_u$ above $p$ is essentially the same as $\bbQ^*_f(K^*,\Sigma^*)$ above
$p^*$ (compare \ref{2q.37}(2)). It should be clear that it is enough to find
a condition $q^*\geq p^*$ with the properties (a)--(c) restated for
$\bbQ^*_f(K,\Sigma)$. 

Let  $\varphi_\bH(i)=|\prod\limits_{j<i}\bH(j)|$. It follows from
\ref{2q.23}(a) that $\varphi_\bH(i)\leq |\pos^u_{i-1}|^i<(\beth(30i+3))^i
<\beth(30i+4)$ and $2^{\varphi_\bH(i)}<\beth(30i+5)$. Therefore, 
\[\begin{array}{l}
2^{\varphi_\bH(i)}\cdot (f(j,i)+\varphi_\bH(i)+2)\leq
\beth(30i+5)\cdot \big(\big(\beth(30i+10)\big)^{j+1}+\beth(30i+4)+ 2\big)<\\ 
\beth(30i+7)\cdot\big(\beth(30i+10)\big)^{j+1}<\big(\beth(30i+10)
\big)^{j+2} =f(j+1,i). 
\end{array}\]
Since plainly $f(j,i)\leq f(j,i+1)$, we conclude that the function $f$ is
$\bH$--fast. Therefore \cite[Theorem 2.2.11]{RoSh:470} gives us a condition
$q^*$ satisfying (a)+(b) (restated for $\bbQ^*_f(K^*,\Sigma^*)$). Proceeding
as in \cite[Theorem 5.1.12]{RoSh:470} but using the large amount of bigness
here (see \ref{2q.23}(d)) we may find a stronger condition saisfying also
demand (c).  

Note that to claim just properness of $\bbQ_u$ one could use the quite
strong halving of $\nor_i$ and \cite{RShS:941}.
\end{proof}

\begin{observation}
\label{2q.16}
\begin{enumerate}
\item $D^{u_1 \cup u_2}_i=D^{u_1}_i\cup D^{u_2}_i$.
\item $h\in\pos^{u_1 \cup u_2}_i$\quad if and only if\quad $h$ is a function
  with domain $D^{u_1 \cup u_2}_i$ and $h\rest D^{u_\ell}_i\in
  \pos^{u_\ell}_i$ for $\ell=1,2$.   
\end{enumerate}
\end{observation}

\begin{definition}
\label{restrdef}
Assume that $\emptyset\neq w\subseteq u\subseteq {\rm Ord}$ are finite,
$v=u\setminus w\neq \emptyset$. Let $\cF\in\wpos^u_i$. We define
$\cF\upharpoonleft w:\pos^w_i\longrightarrow [0,1]$ by
\[(\cF\upharpoonleft w)(h)=\frac{\sum\{\cF(e):h\subseteq
  e\in\pos^u_i\}}{|\pos^v_i|}\qquad \mbox{ for }h\in \pos^w_i.\]
We will also keep the convention that if $u\subseteq {\rm Ord}$ and $\cF\in
\pos^u_i$, then $\cF\upharpoonleft u=\cF$. 
\end{definition}

\begin{proposition}
\label{3.10A}
Assume that $\emptyset\neq u_0\subseteq u_1\subseteq {\rm Ord}$ are finite,
$u_0\neq u_1$ and $\cF_1\in\wpos^{u_1}_i$. Let $\cF_0:=\cF_1\upharpoonleft
u_0$. Then
\begin{enumerate}
\item $\cF_0\in\wpos^{u_0}_i$ and $\frac{\|\cF_0\|}{|\pos^{u_0}_i|}=
  \frac{\|\cF_1\|}{|\pos^{u_1}_i|}$. 
\item If $\cF_2\in\wpos^{u_0}_i$, $\cF_2\leq \cF_0$, then there is
  $\cF_3\in\wpos^{u_1}_i$ such that $\cF_3\leq\cF_1$ and
  $\cF_3\upharpoonleft u_0=\cF_2$.
\end{enumerate}
\end{proposition}

\begin{proof}
Let $v=u_1\setminus u_0$.

\noindent (1)\quad Plainly, $\cF_0\in\wpos^{u_0}_i$. Also
\[\|\cF_0\|=\frac{1}{|\pos^v_i|}\sum\big\{\sum\{\cF_1(e):h\subseteq e \in
\pos^{u_1}_i\}:h\in\pos^{u_0}_i\big\}= \frac{\|\cF_1\|}{|\pos^v_i|}=
\frac{|\pos^{u_0}_i|}{|\pos^{u_1}_i|}\cdot\|\cF_1\|.\] 
\medskip

\noindent (2)\quad Suppose $\cF_2\in\wpos^{u_0}_i$, $\cF_2\leq \cF_0$. For
$e\in \pos^{u_1}_i$ such that $\cF_0(e\upharpoonleft u_0)>0$ we put
\[\cF_3(e)= \cF_1(e)\cdot\frac{\cF_2(e\upharpoonleft u_0)}{\cF_0(e
  \upharpoonleft u_0)},\]
and for $e\in\pos^{u_1}_i$ such that $\cF_0(e\upharpoonleft u_0)=0$ we let
$\cF_3(e)=0$. Then clearly $\cF_3\in\wpos^{u_1}_i$, $\cF_3\leq \cF_1$ and
for $h\in \pos^{u_0}_i$ we have:
\[(\cF_3\upharpoonleft u_0)(h)=\frac{\sum\{\cF_3(e):h\subseteq
  e\in\pos^{u_1}_i\}}{|\pos^v_i|}=\frac{\cF_2(h)}{\cF_0(h)}\cdot
\frac{\sum\{\cF_1(e): h\subseteq e\in\pos^{u_1}_i\}}{|\pos^v_i|}=\cF_2(h).\] 
\end{proof}

\begin{definition}
\label{2q.11}
\begin{enumerate}
\item We say that a pair $(\cF_1,\cF_2)$ is {\em balanced\/} when for some 
$i<\omega$ and finite non-empty sets $u_1,u_2\subseteq {\rm Ord}$ we have
$\cF_\ell\in\wpos^{u_\ell}_i$ for $\ell=1,2$ and $\|\cF_1\|/|\pos^{u_1}_i|=
\|\cF_2\|/|\pos^{u_2}_i|$ and, moreover, if $u_1\cap u_2\neq \emptyset$ then
also $\cF_1\upharpoonleft (u_1\cap u_2)=\cF_2 \upharpoonleft (u_1\cap
u_2)$. 
\item A pair $(\cF_1,\cF_2)$ is {\em strongly balanced\/} if it is balanced
  and $0\neq |u_1\setminus u_2| = |u_2\setminus u_1|$ (where $\cF_\ell\in
  \wpos^{u_\ell}_i$ for $\ell=1,2$).
\item Assume $\cF_\ell\in \wpos^{u_\ell}_i$ (for $\ell=1,2$). Let $u=u_1\cup
  u_2$. We define $\cF=\cF_1*\cF_2\in\ypos^{u_1\cup u_2}_i$ (see
  \ref{2q.8}(I)($\gamma$)) by putting for $h\in\pos^{u_1\cup u_2}_i$
\[\cF(h)=\cF_1(h\upharpoonleft u_1)\cdot \cF_2(h\upharpoonleft u_2).\]
\end{enumerate}
\end{definition}

\begin{remark}
\label{remarkx}
\begin{enumerate}
\item Note that $\cF_1*\cF_2$ can be constantly zero, so it does not have to
  be a member of $\wpos$. However, below we will apply to it our notation
  and definitions formulated for $\wpos$.
\item If $\cF_\ell\in\wpos^{u_\ell}_i$ ($\ell=1,2$), $u_0= u_1\cap u_2\neq
  \emptyset$, and $\cF_3=\cF_1 *\cF_2$, then $\cF_3\upharpoonleft u_0=
  (\cF_1\upharpoonleft u_0)\cdot (\cF_2\upharpoonleft u_0)$. 
\item If $u_1\cap u_2=\emptyset$, $\cF_\ell\in\wpos^{u_\ell}_i$, then
  $\|\cF_1*\cF_2\| =\|\cF_1\|\cdot \|\cF_2\|$. 
\item Suppose $(\cF_1,\cF_2)$ is balanced, $\cF_\ell\in\wpos^{u_\ell}_i$
  (for $\ell=1,2$). Choose finite $u'_1,u'_2\subseteq {\rm Ord}$ such that
  $u_1\subseteq u'_1$, $u_2\subseteq u'_2$, $u_1\cap u_2=u'_1\cap u'_2$ and
  $|u'_1\setminus u'_2|=|u'_2\setminus u'_1|\neq 0$. For $\ell=1,2$ and
  $h\in\pos^{u'_\ell}_i$ put $\cF'_\ell(h)=\cF_\ell(h\upharpoonleft
  u_\ell)$. Then $(\cF'_1,\cF'_2)$ is strongly balanced and $\cF'_\ell
  \upharpoonleft u_\ell=\cF_\ell$. 
\end{enumerate}
\end{remark}

\begin{proposition}
\label{x38}
\begin{enumerate}
\item If $(u_1,u_2)$ is a $\Delta$--system pair, $u_1\neq
  u_2\neq\emptyset$, $\cF_\ell\in\wpos^{u_\ell}_i$ for $\ell=1,2$, and
  $\cF_2=\OP_{u_2,u_1}(\cF_1)$, then the pair $(\cF_1,\cF_2)$ is strongly
  balanced.  
\item If $\cF_\ell\in\wpos^{u_\ell}_i$ for $\ell=1,2$ and
  $\|\cF_\ell\|/|\pos^{u_\ell}_i| \ge a > 0$, the pair $(\cF_1,\cF_2)$ is
  balanced, $u_3=u_1\cup u_2$ and $\cF=:\cF_1*\cF_2$, then
  $\|\cF\|/|\pos^{u_3}_i| \ge \frac{a^3}{8}$.  
\end{enumerate}
\end{proposition}

\begin{proof}
(1)\quad  Straightforward.

\noindent (2)\quad Let $u_0=u_1\cap u_2$. We may assume $u_0\neq \emptyset$
(see \ref{remarkx}(3)). Let $\cF_3:=\cF$ and $\cF_0=\cF_1\upharpoonleft u_0
=\cF_2 \upharpoonleft u_0$. For $h\in\pos^{u_0}_i$ and $\ell \le 3$ let
$\cF^{[h]}_\ell:\pos^{u_\ell}_i\longrightarrow [0,1]$ be defined by 
\[\cF^{[h]}_\ell(e)=\left\{\begin{array}{ll}
\cF_\ell(e)&\mbox{ if }h\subseteq e,\\
0 &\mbox{ otherwise.}
\end{array}\right.\]
Note that
\begin{enumerate}
\item[$(*)_0$]  $k_\ell = |\{e\in\pos^{u_\ell}_i: h\subseteq e\}|$ for
  $h\in\pos^{u_0}_i$, $\ell=1,2$, i.e., this number does not depend on $h$.  
\end{enumerate}
[Why? By the definition of $\pos^{u_\ell}_i$ and \ref{2q.16}.]  
\begin{enumerate}
\item[$(*)_1$]  $\cF_\ell$ is the disjoint sum of $\langle \cF^{[h]}_\ell:h
  \in\pos^{u_0}_i \rangle$ for $\ell=1,2,3$; the ``disjoint'' means that
  $\langle \set(\cF^{[h]}_\ell):h\in\pos^{u_0}_i\rangle$ are pairwise
  disjoint. Hence $\|\cF_\ell\|=\sum\{\|\cF^{[h]}_\ell\|:
  h\in\pos^{u_0}_i\}$. 
\end{enumerate}
[Why?  By the definition of $\pos^{u_\ell}_i$ and $\cF^{[h]}_\ell$.] 
\begin{enumerate}
\item[$(*)_2$]  $k_\ell\geq \|\cF^{[h]}_\ell\|=\cF_0(h)\cdot k_\ell$ for
  $\ell=1,2$.   
\end{enumerate}
[Why? By Defintion \ref{restrdef}.] 
\begin{enumerate}
\item[$(*)_3$]  $\|\cF^{[h]}_3\|=\|\cF^{[h]}_2\|\times\|\cF^{[h]}_1\|$. 
\end{enumerate}
[Why?  By the choice of $\cF^{[h]}_3$.]  

Let (noting that $0 < a \le 1$)
\begin{enumerate}
\item[$(*)_4$]  $A_0=\{h\in\pos^{u_0}_i:\cF_0(h)\geq\frac{a}{2}\}$. 
\end{enumerate}
Now
\begin{enumerate}
\item[$(*)_5$] $|A_0|\geq \frac{a}{2-a}\times |\pos^{u_0}_i|$.
\end{enumerate}
[Why?  Letting $d =|A_0|/|\pos^{u_0}_i|$ and $b =\frac{a}{2}$ (so $0<b\leq
\frac{1}{2}$) we have 
\[h\in \pos^{u_0}_i \setminus A_0\quad \Rightarrow\quad
\|\cF^{[h]}_1\| \le \frac{a}{2} k_1 =bk_1\] 
(remember $(*)_2$). Also $\|\cF^{[h]}_1\|\leq k_1$ for all
$h\in\pos^{u_0}_i$ and $k_1\cdot |\pos^{u_0}_i|=|\pos^{u_1}_i|$. Hence 
\[\begin{array}{l}
a\times |\pos^{u_1}_i|\leq\|\cF_1\|=\sum\{\|\cF^{[h]}_1\|:h\in\pos^{u_0}_i
\}= \\  
\sum\{\|\cF^{[h]}_1\|:h\in\pos^{u_0}_i\setminus A_0\}+\sum\{\|
\cF^{[h]}_1\|:h \in A_0\}\leq bk_1\cdot (|\pos^{u_0}_i|-|A_0|)+
k_1|A_0|= \\ 
bk_1(1-d)|\pos^{u_0}_i| + k_1 d|\pos^{u_0}_i|= k_1\cdot |\pos^{u_0}_i|\cdot
(b(1-d) + d) = |\pos^{u_1}_i|(b + (1-b)d). 
\end{array}\]
Hence $a \le b + (1-b)d$ and $\frac{a-b}{1-b} \le d$. So, as $b = a/2$, we
have $d \ge \frac{a/2}{1-a/2} = \frac{a}{2-a}$. By the choice of $d$ we
conclude $|A_0| = d \times |\pos^{u_0}_i|\geq \frac{a}{2-a}\times
|\pos^{u_0}_i|$, i.e., $(*)_5$ holds.]

Now
\begin{enumerate}
\item[$(*)_6$] $\|\cF_3\|\ge\frac{a^2}{4}\times k_1 \times k_2\times |A_0|$.   
\end{enumerate}
[Why? By $(*)_3$, $\|\cF^{[h]}_3\|=\|\cF^{[h]}_1\|\times \|\cF_2^{[h]}\|$
for all $h\in \pos^{u_0}_i$ and hence  
\[\begin{array}{l}
\|\cF_3\|=\sum\{\|\cF^{[h]}_3\|:h\in\pos^{u_0}_i\}=
\sum\{\|\cF^{[h]}_1\|\times \|\cF_2^{[h]}\|:h\in \pos^{u_0}_i\}\geq\\
\sum\{\|\cF^{[h]}_1\|\times \|\cF_2^{[h]}\|:h\in A_0\} \ge 
\sum\{\frac{a^2}{4}\cdot k_1\cdot k_2:h\in A_0\} =\frac{a^2}{4}\cdot
k_1\cdot k_2 \cdot |A_0|.  
\end{array}\]
So $(*)_6$ holds.]

Lastly,
\begin{enumerate}
\item[$(*)_7$] $\|\cF_3\|\ge\frac{a^3}{8}|\pos^{u_3}_i|$.
\end{enumerate}
Why? Note that $k_1\cdot k_2\cdot |\pos^{u_0}_i|=|\pos^{u_3}_i|$ and hence 
\[\begin{array}{l}
\|\cF_3\|\ge\frac{a^2}{4}\times k_1\times k_2 \times |A_0| =
\frac{a^2}{4}(|A_0|/|\pos^{u_0}_i|)(k_1\times k_2\times |\pos^{u_0}_i|)=\\  
\frac{a^2}{4}\times(|A_0|/|\pos^{u_0}_i|)\times |\pos^{u_3}_i|\geq 
\frac{a^2}{4}\times\frac{a}{2-a}\times |\pos^{u_3}_i|\geq \frac{a^3}{8}
|\pos^{u_3}_i|.
\end{array}\]
So $(*)_7$ holds and we are done. 
\end{proof}

\begin{remark} 
  In \ref{x38}(2) we can get a better bound, the proof gives
  $\frac{a^4}{4(2-a)^2}$ and we can point out the minimal value, gotten when
  all are equal.
\end{remark}

\begin{definition}
\label{projdef}
Let $\bbP,\bbQ$ be forcing notions. 
\begin{enumerate}
\item A mapping $\bj:\bbP\longrightarrow \bbQ$ is called {\em a projection
    of $\bbP$ onto $\bbQ$\/} when: 
\begin{enumerate}
\item[(a)] $\bj$ is ``onto'' $\bbQ$ and 
\item[(b)] $p_1 \le_{\bbP} p_2\quad \Rightarrow\quad \bj(p_1)\le_{\bbQ}
  \bj(p_2)$. 
\end{enumerate}
\item A projection $\bj:\bbP\longrightarrow\bbQ$ is {\em
    $\lessdot$--complete\/} if (in addition to (a), (b) above):
\begin{enumerate}
\item[(c)]  if $\bbQ\models$`` $\bj(p)\le q$ '', then some $p_1$ satisfies
  $p\le_{\bbP} p_1$ and $q \le_{\bbQ}\bj(p_1)$. 
\end{enumerate}
\end{enumerate}
\end{definition}

\begin{definition}
\label{2q.48}
If $\emptyset\neq u\subseteq v\subset\text{Ord}$ are finite, then
$\bj_{u,v}$ is a function from $\underline{\bbQ}_v$ onto $\underline{\bbQ}_u$
defined by:  

for $q\in\underline{\bbQ}_v$ we have $\bj_{u,v}(q)=p\in\underline{\bbQ}_u$ if
and only if 
\begin{enumerate}
\item[$(\alpha)$]  $\bi(p)=\bi(q)$ and $\bar{x}_p=\bar{x}_q \upharpoonleft 
  u$, and 
\item[$(\beta)$] for $i\in [\bi(p),\omega)$ we have $\gc^p_i
  :=\proj_u(\gc^q_i)$ which means $\gc^p_i =(\cF_{\gc^q_i}\upharpoonleft u,
  m_{\gc^p_i})$.   
\end{enumerate}
\end{definition}

\begin{proposition}
\label{2q.45}
If $u\subseteq v\in\mbox{\rm Ord}^{<\aleph_0}$, then $\bj_{u,v}$ is a (well
defined) $\lessdot$--complete projection from $\underline{\bbQ}_v$ onto
$\underline{\bbQ}_u$. 
\end{proposition} 

\begin{proof}
It follows from \ref{3.10A} that
\begin{enumerate}
\item[$(*)_1$] if $\gc\in\underline{\CR}^v_i$, then $\proj_u(\gc)\in
  \underline{\CR}^u_i$ and $\nor_i(\proj_u(\gc))=\nor_i(\gc)$.
\end{enumerate}
Also, by the definition of $\proj_u$ and \ref{restrdef}, easily 
\begin{enumerate}
\item[$(*)_2$] if $\gc\in\underline{\CR}^v_i$,
  $\gd\in\underline{\Sigma}(\gc)$, then $\proj_u(\gd)\in
  \underline{\Sigma}(\proj_u(\gc))$, and 
\item[$(*)_3$] if $\gd\in\underline{\CR}^u_i$, $\cF:\pos^v_i\longrightarrow
  [0,1]$ is defined by $\cF(h)=\cF_\gd(h\upharpoonleft u)$, then
  $(\cF,m_\gd)\in \underline{\CR}^v_i$,
  $\nor_i\big((\cF,m_\gd)\big)=\nor_i(\gd)$ and
  $\proj_u\big((\cF,m_\gd)\big)=\gd.$ 
\end{enumerate}
Therefore $\bj_{u,v}$ is a projection from $\underline{\bbQ}_v$ onto
$\underline{\bbQ}_u$. To show that it is $\lessdot$--complete we note that,
by \ref{3.10A}(2),
\begin{enumerate}
\item[$(*)_4$] if $\gc_1\in\underline{\CR}^v_i$, $\gc_0=\proj_u(\gc_1)$ and
  $\gc_2\in \underline{\Sigma}(\gc_0)$, then some $\gc_3\in
  \underline{\CR}^v_i$ satisfies $\gc_3\in \underline{\Sigma}(\gc_1)$ and
  $\proj_u(\gc_3)=\gc_2$. 
\end{enumerate}
The rest should be clear.
\end{proof}

\begin{proposition}
\label{2q.52}
Assume $(u_1,u_2)$ is a $\Delta$--system pair, i.e.,  $u_1,u_2 \subseteq
\mbox{\rm Ord}$, $|u_1| = |u_2| < \aleph_0$ and so $\OP_{u_2,u_1}$ (the
order isomorphism from $u_1$ onto $u_2$, see \ref{0z.1}(10)) is the identity
on $u_1 \cap u_2$. Let $u= u_1 \cup u_2$. Further assume that $p_\ell \in 
\underline{\bbQ}_{u_\ell}$ for $\ell=1,2$, $\nor^1_i(\gc^{p_\ell}_i) \geq 1$
for all $i\geq \bi(p_\ell)$ and $\OP_{u_1,u_2}$ maps $p_1$ to $p_2$. Then
there is a condition $q\in\underline{\bbQ}_u$ such that:  
\begin{enumerate}
\item[(a)] $\bi(q)=\bi(p_1)$ and $p_\ell\leq_{\underline{\bbQ}_{u_\ell}}
  \bj_{u_\ell,u}(q)$ for $\ell=1,2$, and 
\item[(b)] $\nor^1_i(\gc^q_i)\ge \nor^1_i(\gc^{p_1}_i)-1$ for $i\in
  [\bi(q),\omega)$. 
\end{enumerate}
\end{proposition}

\begin{proof}  
We shall mainly use clause (2) of \ref{x38}.  

First, we set $\bi(q)=\bi(p_1)$ and we let $\bar{x}=\langle
x_i:i<\bi(q)\rangle$, where $x_i = (f_{x_i},g_{x_i},e_{x_i})$ is defined by 
\begin{enumerate}
\item[$(\bullet_1)$] $f_{x_i} = f_{x^{p_1}_i}\cup f_{x^{p_2}_i}$, it is well
  defined function because $f_{x^{p_\ell}_i}\in {}^{u_\ell}(
  \Per({}^{n_*(i)}2))$ for $\ell=1,2$ are well defined functions, with the 
  same restriction to $u_0 = u_1 \cap u_2$; 
\item[$(\bullet_2)$] $g_{x_i}=g_{x^{p_1}_i}\cup g_{x^{p_2}_i}$ (similarly
  well defined);
\item[$(\bullet_3)$] $e_{x_i}=e_{x^{p_1}_i}\cup e_{x^{p_2}_i}$ (again, it is
  well defined). 
\end{enumerate}
Easily,
\begin{enumerate}
\item[$(\bullet_4)$] $\bar{x}\in\bS_{u,\bi(q)}$.
\end{enumerate}
Second, we let $\bar{\gc}=\langle \gc_i:i\in [\bi(q),\omega)\rangle$ where
for $i\in [\bi(q),\omega)$ we let $\gc_i= (\cF_i,m_i)$, where
\begin{enumerate}
\item[$(\bullet_5)$] $\cF_i=\cF_{\gc^{p_1}_i}*\cF_{\gc^{p_2}_i}$,
\item[$(\bullet_6)$] $m_i=m_{\gc^{p_\ell}_i}$ for $\ell=1,2$.
\end{enumerate}
Let $i\in [\bi(q),\omega)$. By Proposition \ref{x38}(1) we know that the
pair $(\cF_{\gc^{p_1}_i},\cF_{\gc^{p_2}_i})$ is (strongly) balanced. Let
$a=\frac{\|\cF_{\gc^{p_1}_i}\|}{|\pos^{u_1}_i|}=\frac{\|\cF_{\gc^{p_2}_i}\|}{|
  \pos^{u_2}_i|}$. Then, by \ref{x38}(2) we have $\|\cF_i\|\geq
\frac{a^3}{8}\times |\pos^u_i|$. Hence, recalling $k^*_i\geq 3$,
\[\hspace{-5pt}\begin{array}{l}
\nor^0_i(\cF_i)=k^*_i-\log_3\big(\log_{k^*_i}\big(\frac{k^*_i\cdot
  |\pos^u_i|}{\|\cF_i\|}\big)\big)\geq k^*_i-
\log_3\big(\log_{k^*_i}\big(\frac{8k^*_i}{a^3}\big)\big)\geq
k^*_i-\log_3\big(3\log_{k^*_i}\big(\frac{k^*_i}{a}\big)\big) =\\
k^*_i-\log_3\big(\log_{k^*_i}\big(\frac{k^*_i\cdot
  |\pos^{u_1}_i|}{\|\cF_{\gc^{p_1}_i}\|}\big)\big)-1=\nor_i^0(\cF_{\gc^{p_1}_i})-1
  = \nor_i^0(\cF_{\gc^{p_2}_i})-1.
\end{array}\] 
Now clearly $q:=(\bar{x},\bar{\gc})$ is as required. 
\end{proof}

\section{Definable branches and disjoint cones} 
Now we come to the claim on creatures specifically to deal with the bounded
intersection of branches.  We think below of $H_\ell$ as part of a name of a
branch of the $\alpha$-th tree. 

\begin{lemma}
\label{3c.30}
Assume that $u = u_1\cup u_2$ are finite non-empty sets of ordinals, $|u_2
\setminus u_1|=|u_1\setminus u_2|\neq 0$, $w=u_1\cap u_2$. Suppose also that 
$i=j+1<\omega$, $\cF_\ell\in\wpos^{u_\ell}_i$ (for $\ell=1,2$) and the pair
$(\cF_1,\cF_2)$ is balanced. Let $S$ be a finite set (e.g., ${}^{n_*(i)}2$)
and $H_\ell:\pos^{u_\ell}_i\longrightarrow S$. Then there are
$\cF'_1,\cF'_2,\cF$ such that:  
\begin{enumerate}
\item[(a)]  $\cF\in\wpos^u_i$,
\item[(b)]  $\cF'_\ell\leq\cF_\ell$ for $\ell=1,2$ and $\cF=\cF'_1*\cF'_2$,
\item[(c)]  the pair $(\cF'_1,\cF'_2)$ is balanced, 
\item[(d)]  $\|\cF'_\ell\|\ge\frac{1}{8}\|\cF_\ell\|$ for $\ell=1,2$, 
\item[(e)] one of the following occurs:
\begin{enumerate}
\item[$(\alpha)$] if $h\in\set(\cF)$ then $H_1(h\upharpoonleft u_1) \ne H_2(h
  \upharpoonleft u_2)$, 
\item[$(\beta)$] {\em (Case 1)\/}  $u_1 \cap u_2 = \emptyset$: for some
  $s\in S$ we have $h\in\set(\cF) \Rightarrow H_1(h\upharpoonleft u_1)=s= 
  H_2(h\upharpoonleft u_2)$;\\ 
{\em (Case 2)\/} general: for some function $H'$ from  $\pos^w_i$ to $S$ we
have:  
\[h\in\set(\cF)\quad\Rightarrow\quad H_1(h\upharpoonleft u_1)=
H'(h\upharpoonleft (u_1\cap u_2))=H_2(h\upharpoonleft u_2).\]
\end{enumerate}
\end{enumerate}
\end{lemma}

\begin{proof}  
Let $\langle s_m:m<m_*\rangle$ list of all members of $S$. Let $g\in\cG :=
\pos^w_i$. Now for every $m \le m_*$ we define 
\begin{enumerate}
\item[$(\oplus_1)$]
\begin{enumerate}
\item[(a)] $\cF_{\ell,g}:\pos^{u_\ell}_i\longrightarrow [0,1]$ is given by
  $\cF_{\ell,g}(h)=\cF_\ell(h)$ if $g\subseteq h$ and $\cF_{\ell,g}(h)=0$ otherwise, 
\item[(b)] $k_{\ell,g}:= \|\cF_{\ell,g}\|$,
\item[(c)] $k^<_{\ell,m,g}:=\sum\big\{\cF_{\ell,g}(h):g\subseteq
  h\in\pos^{u_\ell}_i\ \&\ H_\ell(h)\in\{s_{m_1}: m_1 <m\}\big\}$, 
\item[(d)] $k^=_{\ell,m,g}:=\sum\big\{\cF_{\ell,g}(h):g\subseteq
  h\in\pos^{u_\ell}_i\ \&\ H_\ell(h)=s_m\big\}$, 
\item[(e)] $k^\ge_{\ell,m,g}:=\sum\big\{\cF_{\ell,g}(h):g\subseteq
  h\in\pos^{u_\ell}_i\ \&\ H_\ell(h)\in\{s_{m_1}: m\leq m_1<m_*\}\big\}$. 
\end{enumerate}
\end{enumerate}
Since we are assuming that $(\cF_1,\cF_2)$ is strongly balanced, we have 
\begin{enumerate}
\item[$(\oplus_2)$] $k_{1,g}=k_{2,g}$, call it $k_g$.
\end{enumerate}
Plainly, $k^<_{\ell,m,g},k^=_{\ell,m,g},k^\geq_{\ell,m,g},k_g$ are
non-negative reals and 
\begin{enumerate}
\item[$(*)_1$]  $k^<_{\ell,m,g} + k^\ge_{\ell,m,g} = k_g$.
\end{enumerate}
Hence
\begin{enumerate}
\item[$(*)_2$]  $\max\{k^<_{\ell,m,g},k^\ge_{\ell,m,g}\} \ge k_g/2$.
\end{enumerate}
Also,
\begin{enumerate}
\item[$(*)_3$]  $k^<_{\ell,m,g} \le k^<_{\ell,m+1,g}$ and $k^\ge_{\ell,m,g}
  \ge k^\ge_{\ell,m+1,g}$, in fact $k^<_{\ell,m,g}+k^=_{\ell,m,g}=
  k^<_{\ell,m+1,g}$ and $k^\geq_{\ell,m+1,g}+k^=_{\ell,m,g}=
  k^\geq_{\ell,m,g}$, and 
\item[$(*)_4$] $k^<_{\ell,0,g}=0=k^\ge_{\ell,m_*,g}$.
\end{enumerate}
Hence for some $m_{\ell,g}$ we have
\begin{enumerate}
\item[$(*)_5$] $k^<_{\ell,m_{\ell,g} +1,g}\ge k_g/2$ and
  $k^\ge_{\ell,m_{\ell,g},g}\ge k_g/2$.  
\end{enumerate}
Therefore:
\begin{enumerate}
\item[$(*)_6$]  one of the following possibilities holds:
\begin{enumerate}
\item[(a)]  both $k^<_{\ell,m_{\ell,g},g}$ and $k^\ge_{\ell,m_{\ell,g}
    +1,g}$ are greater than or equal to $k_g/4$, or 
\item[(b)] $k^=_{\ell,m_{\ell,g},g}\ge k_g/4$. 
\end{enumerate}
\end{enumerate}
[Why?  If clause (b) fails then by $(*)_5$ we get clause (a).]

Choose $(\iota_g,\cF^*_{1,g},\cF^*_{2,g})$ as follows.
\begin{enumerate}
\item[$(*)_7$] {\em Case 1:\/}  $k^=_{1,m_{1,g},g}\ge k_g/4$ and
  $k^=_{2,m_{2,g},g} \ge k_g/4$.

Let $\iota_g=1$, and $\cF^*_{\ell,g}:\pos^{u_\ell}_i\longrightarrow [0,1]$
be such that $\cF^*_{\ell,g}(h)=\cF_{\ell,g}(h)$ if $g\subseteq h$ and
$H_\ell(h)=s_{m_{\ell,g}}$, and $\cF^*_{\ell,g}(h)=0$ otherwise (for
$\ell=1,2$). 

{\em Case 2:\/} $k^=_{1,m_{1,g},g}\ge k_g/4$ and $k^=_{2,m_{2,g},g}<k_g/4$. 

Let $\iota_g=2$ and $\cF^*_{\ell,g}:\pos^{u_\ell}_i\longrightarrow [0,1]$
(for $\ell=1,2$) be defined by:\\
$\cF^*_{1,g}(h)=\cF_{1,g}(h)$ if $g\subseteq h$ and
$H_1(h)=s_{m_{1,g}}$, and $\cF^*_{1,g}(h)=0$ otherwise;\\
$\cF^*_{2,g}(h)=\cF_{2,g}(h)$ if $g\subseteq h$ and
$H_2(h)\neq s_{m_{1,g}}$, and $\cF^*_{2,g}(h)=0$ otherwise.

{\em Case 3:\/} $k^=_{1,m_{1,g},g}<k_g/4$ and $k^=_{2,m_{2,g},g}\ge k_g/4$.  

Let $\iota_g=3$ and $\cF^*_{\ell,g}:\pos^{u_\ell}_i\longrightarrow [0,1]$
(for $\ell=1,2$) be defined by:\\
$\cF^*_{1,g}(h)=\cF_{1,g}(h)$ if $g\subseteq h$ and
$H_1(h)\neq s_{m_{2,g}}$, and $\cF^*_{1,g}(h)=0$ otherwise;\\
$\cF^*_{2,g}(h)=\cF_{2,g}(h)$ if $g\subseteq h$ and
$H_2(h)=s_{m_{2,g}}$, and $\cF^*_{2,g}(h)=0$ otherwise.

{\em Case 4:\/}  $k^=_{1,m_{1,g},g}<k_g/4$, $k^=_{2,m_{2,g},g}<k_g/4$ and
$m_{1,g} \le m_{2,g}$. 

Let $\iota_g = 4$ and $\cF^*_{\ell,g}:\pos^{u_\ell}_i\longrightarrow [0,1]$
(for $\ell=1,2$) be defined by:\\
$\cF^*_{1,g}(h)=\cF_{1,g}(h)$ if $g\subseteq h$ and $H_1(h) \in\{s_0,\ldots,
s_{m_{1,g}-1}\}$, and $\cF^*_{1,g}(h)=0$ otherwise;\\  
$\cF^*_{2,g}(h)=\cF_{2,g}(h)$ if $g\subseteq h$ and
$H_2(h)\in\{s_{m_{1,g}},\ldots,s_{m_*-1}\}$, and $\cF^*_{2,g}(h)=0$ otherwise.

{\em Case 5:\/} $k^=_{1,m_{1,g},g}<k_g/4$, $k^=_{2,m_{2,g},g}<k_g/4$ and
$m_{1,g} > m_{2,g}$. 

Let $\iota_g=5$ and $\cF^*_{\ell,g}:\pos^{u_\ell}_i\longrightarrow [0,1]$
(for $\ell=1,2$) be defined by:\\
$\cF^*_{1,g}(h)=\cF_{1,g}(h)$ if $g\subseteq h$ and $H_1(h)
\in\{s_{m_{2,g}},\ldots,s_{m_*-1}\}$, and $\cF^*_{1,g}(h)=0$ otherwise;\\   
$\cF^*_{2,g}(h)=\cF_{2,g}(h)$ if $g\subseteq h$ and
$H_2(h)\in\{s_0,\ldots,s_{m_{2,g}-1}\}$, and $\cF^*_{2,g}(h)=0$ otherwise. 
\end{enumerate}
Now:
\begin{enumerate}
\item[$(*)_8$]  $\|\cF^*_{\ell,g}\|\ge\frac{1}{4}\|\cF_{\ell,g}\|=\frac{1}{4}
  k_g$ for $\ell=1,2$. 
\end{enumerate}
[Why? By $(\oplus_2)$ and $(*)_7$ - check each case.]

Finally choose $\cF^{**}_{\ell,g}$ (for $\ell=1,2$ and $g\in \cG$) such that: 
\begin{enumerate}
\item[$(*)_9$] 
\begin{enumerate}
\item[(a)] $\cF^{**}_{\ell,g}\leq \cF^*_{\ell,g}$, $\|\cF^{**}_{\ell,g}\|
  \ge \frac{1}{4}k_g$, and $\|\cF^{**}_{1,g}\|=\|\cF^{**}_{2,g}\|$,  
\item[(b)] {\em if\/} $(\iota_g=1\wedge m_{1,g} = m_{2,g})$ {\em then\/} for
  some $s=s(g)\in S$  
\[h_1\in\set(\cF^{**}_{1,g})\wedge h_2\in\set(\cF^{**}_{2,g})\quad
\Rightarrow \quad H_1(h_1) = H_2(h_2)=s,\]
\item[(c)] {\em if\/} $(\iota_g\ne 1\vee m_{1,g}\ne m_{2,g})$ {\em then\/} 
\[h_1\in\set(\cF^{**}_{1,g})\wedge h_2\in\set(\cF^{**}_{2,g})\quad
\Rightarrow \quad H_1(h_1) \ne H_2(h_2).\]  
\end{enumerate}
\end{enumerate}
[Why possible?  We can choose them to satisfy clause (a) by $(*)_8$ and
clauses (b),(c) follow - look at the choices inside $(*)_7$.]   

Now we stop fixing $g\in \cG$. Put
\[\cG^1=\{g\in\cG:\iota_g=1\mbox{ and }m_{1,g} = m_{2,g}\}\quad \mbox{ and } 
\quad \cG^2=\{g\in\cG:\iota_g\ne 1\mbox{ or }m_{1,g}\ne m_{2,g}\}.\] 
When we vary $g\in\cG$, obviously
\begin{enumerate}
\item[$(\circledast_1)$] $\cF_\ell$ is the disjoint sum of $\langle
  \cF_{\ell,g}:g \in \cG\rangle$,
\end{enumerate} 
and hence
\begin{enumerate}
\item[$(\circledast_2)$] $\|\cF_\ell\|=\sum\{k_g:g\in\cG\}$.
\end{enumerate}
As $\cG=\pos^w_i$ is the disjoint union of $\cG^1,\cG^2$, plainly
\begin{enumerate}
\item[$(\circledast_3)$] for some $\cG'\in\{\cG^1,\cG^2\}$ the following
  occurs: 
\[\sum\{k_g:g\in\cG'\}\ge\|\cF_1\|/2=\|\cF_2\|/2.\]
\end{enumerate}
Lastly, we put $\cF'_\ell=\sum\{\cF^{**}_{\ell,g}: g\in \cG'\}$ (for
$\ell=1,2$). We note that  
\[\|\cF'_\ell\|=\sum\{\|\cF^{**}_{\ell,g}\|:g\in \cG'\}\ge \sum\{\frac{1}{4}
k_g:g\in \cG'\}\ge\frac{1}{4}(\|\cF_\ell\|/2)= \frac{1}{8}\|\cF_\ell\|.\] 
Now it should be clear that $\cF'_1,\cF'_2$ and $\cF=\cF'_1*\cF'_2$ are as
required. 
\end{proof}

\begin{crlem}
\label{3c.26}
Assume that
\begin{enumerate}
\item[(a)] $u_1,u_2$ are finite subsets of {\rm Ord}, $|u_1\setminus
  u_2|=|u_2\setminus u_1|\neq 0$,
\item[(b)] $\cF_\ell\in\wpos^{u_\ell}_i$, $i<\omega$ and $\|\cF_\ell\|\geq a 
  \times |\pos^{u_\ell}_i|>0$, 
\item[(c)] $H_\ell$ is a function from $\bS_{u_\ell,i+1}$ to ${}^{n_*(i)}2$, 
\item[(d)] the pair $(\cF_1,\cF_2)$ is balanced.
\end{enumerate}
Let $u = u_1 \cup u_2$ and $w = u_1 \cap u_2$ and $|u|<n_*(i-1)$. Then we
can find $\cF'_\ell\in\wpos^{u_\ell}_i$ and partial functions $\bh_\ell$
from $\bS_{u_\ell,i} \times\bS_{w,i+1}$ into ${}^{n_*(i)}2$ for $\ell=1,2$
and $\cF\in\wpos^u_i$ such that:  
\begin{enumerate}
\item[$(\alpha)$] $\cF'_\ell\leq\cF_\ell$, $\|\cF'_\ell\|\geq
  8^{-k_*}\|\cF_\ell\|$, where $k_*=|\bS_{u,i}|<\ell^*_i$, and the pair
  $(\cF'_1,\cF'_2)$ is balanced, 
\item[$(\beta)$]  $\cF=\cF'_1*\cF'_2$ and so $\cF\upharpoonleft
  u_\ell\leq\cF_\ell$ for $\ell=1,2$ and $\|\cF\|/|\pos^u_i|\ge 
  \frac{a^3}{2^{9k_*+3}}$,  
\item[$(\gamma)$] if $h\in\set(\cF)$, $\bar{x}\in\bS_{u,i}$ (so
  $\lh(\bar{x})= i$) and $\bar{y}=\suc_{\bar{x}}(h)\in \bS_{u,i+1}$, then 
\[H_1(\bar{y}\upharpoonleft u_1)=H_2(\bar{y}\upharpoonleft u_2)\quad
\Rightarrow \quad \bh_1(\bar{x}\upharpoonleft u_1,\bar{y} \upharpoonleft w)
= \bh_2(\bar{x}\upharpoonleft u_2,\bar{y} \upharpoonleft w)=
H_1(\bar{y}\upharpoonleft u_1)=H_2(\bar{y}\upharpoonleft u_2).\]
\item[$(\delta)$] moreover, for each $\bar{x} \in \bS_{u,i}$ the truth value 
  of the equality $H_1(\bar{y}\upharpoonleft u_1)=H_2(\bar{y}\upharpoonleft
  u_2)$ in clause $(\gamma)$ is the same for all $h\in\set(\cF)$.
\end{enumerate}
\end{crlem}

\begin{proof} 
Let $\langle\bar{x}_k:k<k_*\rangle$ list $\bS_{u,i}$ (without
repetitions). We choose $(\cF_k,\cF_{1,k},\cF_{2,k})$ by induction on
$k\le k_*$ such that: 
\begin{enumerate}
\item[(i)] $\cF_{\ell,k}\in\wpos^{u_\ell}_i$ for $\ell=1,2$, 
\item[(ii)]  if $k=0$, then $\cF_{\ell,k} = \cF_\ell$,
\item[(iii)] $\cF_{\ell,k}$ is $\leq$--decreasing with $k$, i.e.,
  $\cF_{\ell,k+1} \leq \cF_{\ell,k}$, 
\item[(iv)] $\|\cF_{\ell,k}\|\ge\frac{1}{8^k}\|\cF_\ell\|$, 
\item[(v)]  $(\cF_{1,k},\cF_{2,k})$ is balanced,
\item[(vi)]  $\cF_k=\cF_{1,k}*\cF_{2,k}$, so also $\leq$--decreasing with
  $k$,  
\item[(vii)] for each $k$ one of the following occurs:
\begin{enumerate}
\item[$(\alpha)$] {\em if\/} $h\in\set(\cF_{k+1})$ and
  $\bar{y}=\suc_{\bar{x}_k}(h) \in \bS_{u,i+1}$, {\em then\/}
  $H_1(\bar{y}\upharpoonleft u_1)\ne H_2(\bar{y}\upharpoonleft u_2)$;
\item[$(\beta)$] {\em if\/} $h',h''\in\set(\cF_{k+1})$ and $h'\upharpoonleft
  w = h''\upharpoonleft w$, $\bar{y}'=\suc_{\bar{x}_k}(h')$, $\bar{y}''=
  \suc_{\bar{x}_k}(h'')$, {\em then\/} 
\[H_1(\bar{y}'\upharpoonleft u_1)=H_1(\bar{y}''\upharpoonleft u_1)= 
H_2(\bar{y}'\upharpoonleft u_2)=H_2(\bar{y}''\upharpoonleft u_2).\]  
\end{enumerate}
\end{enumerate}
If we carry out the definition then $\cF=\cF_{k_*}$ is as required.  Note
that $\|\cF_{\ell,k_*}\|\geq\frac{\|\cF_\ell\|}{8^{k_*}}$, hence the bound on
$\|\cF\|$, i.e. clause $(\beta)$ of \ref{3c.26} holds by \ref{x38}; that is
we choose $8^{-k_*}a$ here for $a$ there and $\frac{a^3}{8}$ there means
$\frac{(8^{-k_*}a)^3}{8}= \frac{a^3}{2^{9k_*+3}}$ here. 

The initial step of $k=0$ is obvious. For the inductive step, for $k+1$ we
define $H_{\ell,k}$ as follows: for $h \in\pos^{u_\ell}_i$ we put
$H_{\ell,k}(h)=H_\ell(\suc_{\bar{x}_k\upharpoonleft u_\ell}(h))$ and we
apply Lemma \ref{3c.30} to $\cF_{1,k},\cF_{2,k},H_{1,k},H_{2,k}$ here 
standing for $\cF_1,\cF_2,H_1,H_2$ there. This way we obtain
$\cF_{1,k+1},\cF_{2,k+1}$ and we set $\cF_{k+1}=\cF_{1,k+1}*\cF_{2,k+1}$. If
in clause \ref{3c.30}(e) subclause $(\alpha)$ holds, then the demand in
(vii)($\alpha$) is satisfied. Otherwise, we get a function $H'$ such that
for each $h\in\set(\cF_{k+1})$ we have
\[H_{1,k}(h\upharpoonleft u_1)=H'(h\upharpoonleft w)=H_{2,k}(h
\upharpoonleft u_2).\] 
Consequently, the demand in (vii)($\beta$) is fulfilled. Moreover this
choice is O.K. for any $\cF'\subseteq \cF_{k+1}$, so we are done.  
\end{proof}

\begin{lemma}
\label{2q.29}
\begin{enumerate}
\item Assume that $u \subseteq \text{\rm Ord}$ is finite, $\alpha\in u$ and 
$\gc\in\underline{\CR}^u_i$, $i>0$. Suppose also that there are
$\bar{x}\in\bS_{u,i}$ and functions $\bh_1,\bh_2$ such that
\begin{enumerate}
\item[if] $h\in\set(\cF_\gc)$ and $\bar{y}=\suc_{\bar{x}}(h)=
  \bar{x}\conc\langle y\rangle$ (see \ref{2q.8}(I)$(\zeta)$), 
\item[then] $\eta_\ell:=\bh_\ell(h\upharpoonleft (u
  \setminus\{\alpha\}))\in {}^{n_*(i)}2$ is well defined for $\ell=1,2$ and 
  $(g_y(\alpha)^{-1} \circ f_y(\alpha))(\eta_1)=\eta_2$. 
\end{enumerate}
Then $\nor^0_i(\cF_{\gc})=0$.

\item Assume that $w\subseteq u\subseteq \text{\rm Ord}$ are finite,
  $\alpha_1,\alpha_2\in u\setminus w$, $\alpha_1\neq\alpha_2$ and  
$\gc\in\underline{\CR}^u_i$, $i>0$. Suppose also that $\bar{x}\in\bS_{u,i}$
and there are functions $\bh_1,\bh_2$ such that
\begin{enumerate}
\item[if] $h\in\set(\cF_\gc)$ and $\bar{y}=\suc_{\bar{x}}(h)=
  \bar{x}\conc\langle y\rangle$, 
\item[then] $\eta_\ell:=\bh_\ell(\bar{x},\bar{y}\upharpoonleft w)\in
  {}^{n_*(i)}2$ is well defined for $\ell=1,2$ and  
\[(g_y(\alpha_1)^{-1} \circ f_y(\alpha_1))(\eta_1)= (g_y(\alpha_2)^{-1}
\circ f_y(\alpha_2))(\eta_2).\] 
\end{enumerate}
Then $\nor^0_i(\cF_{\gc})=0$.
\end{enumerate}
\end{lemma}

\begin{proof}  
(1)\quad First we try to give an upper bound to
$|\set(\cF_{\gc})|/|\pos^u_i|$. Thinking of ``randomly drawing'' $h_0 \in
\pos^{u \setminus\{\alpha\}}_i$ with equal probability, we get an upper
bound to the fraction of $h\in\pos^u_i$, $h\upharpoonleft (u \setminus
\{\alpha\}) = h_0$ such that if $\suc_{\bar{x}}(h)=\bar{x}\conc \langle
y\rangle$, then   

$\eta_\ell:=\bh_\ell(h\upharpoonleft (u\setminus\{\alpha\}))\in
{}^{n_*(i)}2$ is well defined for $\ell=1,2$ and  $(g^{-1}_y(\alpha) \circ
f_y(\alpha))(\eta_1)=\eta_2$.  

\noindent Since
\[g_y(\alpha)(\nu)\rest n_*(i-1) = g_{x_{i-1}}(\alpha)(\nu\rest n_*(i-1)) =
f_y(\alpha)(\nu) \rest n_*(i-1))\quad\mbox{ for all }\nu \in {}^{n_*(i)}2,\] 
clearly it is $\leq 1/2^{n_*(i)-n_*(i-1)}$. So $\|\cF_\gc\|/|\pos^u_i|\leq 
|\set(\cF_\gc)|/|\pos^u_i| \le 1/2^{n_*(i)-n_*(i-1)}<(k^*_i)^{1-3^{k^*_i}}$
and consequently $\nor^0_i(\cF_{\gc})=0$. 
\medskip

\noindent (2)\quad For $e\in\pos^{u\setminus\{\alpha_1\}}_i$ let $\bar{y}_e=
\suc_{\bar{x}\upharpoonleft (u\setminus\{\alpha_1\})}(e)=
(\bar{x}\upharpoonleft (u\setminus\{\alpha_1\}))\conc \langle y_e\rangle$,
$\bh'_1(e)=\bh_1(\bar{x},\bar{y}_e\upharpoonleft w)$ and $\bh'_2(e)= \big(
g_{y_e}(\alpha_2)^{-1}\circ f_{y_e}(\alpha_2)\big)\big(\bh_2(\bar{x},
\bar{y}_e \upharpoonleft w)\big)$. Since $\alpha_1,\alpha_2\notin w$ and
$\alpha_2\in u\setminus\{\alpha_1\}$, for each $h\in \set(\cF_\gc)$ the
values $\bh'_1\big(h\upharpoonleft (u\setminus\{\alpha_1\})\big),
\bh'_2\big(h\upharpoonleft (u\setminus\{\alpha_1\})\big)$ are well defined
and, letting $\bar{y}=\suc_{\bar{x}}(h)=\bar{x}\conc\langle y\rangle$,
\[\big(g_y(\alpha_1)^{-1}\circ f_y(\alpha_1)\big)\big(\bh'_1(h
\upharpoonleft (u\setminus\{\alpha_1\}))\big) = \bh'_2(h\upharpoonleft
(u\setminus\{\alpha_1\})).\]
Therefore clause (1) applies and $\nor^0_i(\cF_\gc)=0$. 
\end{proof}

Before we state the main corollary to Crucial Lemma \ref{3c.26}, let us
recall that if $\emptyset\neq w\subseteq u$, $\gc\in\underline{\CR}^u_i$,
then $\proj_w(\gc)=(\cF_{\gc}\upharpoonleft w,m_{\gc})\in
\underline{\CR}^w_i$ (see Definition \ref{2q.48}($\beta$)). Also, if
$\emptyset=w =u_1\cap u_2$ and $\gc_\ell\in\underline{\CR}^{u_\ell}_i$, then
$\proj_w(\gc_1)=\proj_w(\gc_2)$ will mean that $\nor_i(\gc_1)=\nor_i(\gc_2)$
and $m_{\gc_1}=m_{\gc_2}$.

\begin{crcor}
\label{cruccor}
Assume that
\begin{enumerate}
\item[(a)] $u_1,u_2$ are finite subsets of {\rm Ord}, $|u_1\setminus
  u_2|=|u_2\setminus u_1|$, $u = u_1 \cup u_2$, $w = u_1 \cap u_2$,
  $\alpha_1\in u_1\setminus u_2$ and $\alpha_2\in u_2\setminus u_1$,
  $1<i<\omega$, $|u|<n_*(i-1)$,  
\item[(b)] $\gc_\ell\in\underline{\CR}^{u_\ell}_i$ and $\nor_i(\gc_\ell)>2$
  (for $\ell=1,2$), and $\proj_w(\gc_1)=\proj_w(\gc_2)$,  
\item[(c)] $H_\ell:\bS_{u_\ell,i+1}\longrightarrow {}^{n_*(i)}2$. 
\end{enumerate}
Then we can find $\gd_\ell\in\underline{\Sigma}(\gc_\ell)$, $\ell=1,2$, such
that:   
\begin{enumerate}
\item[$(\alpha)$] $\proj_w(\gd_1)=\proj_w(\gd_2)$,
\item[$(\beta)$]  $\nor_i(\gd_\ell)\geq \nor_i(\gc_\ell)-1$,
\item[$(\gamma)$] if $h\in\set(\cF_{\gd_1}*\cF_{\gd_2})$,
  $\bar{x}\in\bS_{u,i}$ and $\bar{y}=\suc_{\bar{x}}(h)\in \bS_{u,i+1}$, and 
  $\eta_\ell=H_\ell(\bar{y}\upharpoonleft u_\ell)\in {}^{n_*(i)}2$ (for
  $\ell=1,2$), then 
\[\eta_1=\eta_2\quad \Rightarrow \quad \big(g_{y_i}(\alpha_1)^{-1}\circ
f_{y_i}(\alpha_1)\big)(\eta_1)\neq \big(g_{y_i}(\alpha_2)^{-1}\circ
f_{y_i}(\alpha_2)\big)(\eta_2).\] 
\end{enumerate}
\end{crcor}

\begin{proof}
Let $\cF_\ell=\cF_{\gc_\ell}$. By assumptions (a,b), the pair
$(\cF_1,\cF_2)$ is strongly balanced and $\nor^0_i(\cF_\ell)>
(\ell^*_i)^2$. Apply Crucial Lemma \ref{3c.26} to choose
$\cF'_1,\cF'_2,\bh_1,\bh_2$ such that  
\begin{enumerate}
\item[$(*)_1$] $\cF'_\ell\in\wpos^{u_\ell}_i$, $\cF'_\ell\leq \cF_\ell$,
  $\|\cF'_\ell\|\geq 8^{-k_*}\cdot \|\cF_\ell\|$ (where $k_*=|\bS_{u,i}|$),
  and the pair $(\cF'_1,\cF'_2)$ is balanced,
\item[$(*)_2$] $\bh_\ell:\bS_{u_\ell,i}\times\bS_{w,i+1}\longrightarrow
  {}^{n_*(i)}2$, 
\item[$(*)_3$] if $h\in\set(\cF'_1*\cF'_2)$, $\bar{x}\in \bS_{u,i}$ and
  $\bar{y}=\suc_{\bar{x}}(h)\in\bS_{u,i+1}$, then 
\[H_1(\bar{y}\upharpoonleft u_1)=H_2(\bar{y}\upharpoonleft u_2)\quad
\Rightarrow\quad \bh_1(\bar{x}\upharpoonleft u_1,\bar{y}\upharpoonleft w)=
\bh_2(\bar{x}\upharpoonleft u_2,\bar{y}\upharpoonleft w)=
H_1(\bar{y}\upharpoonleft u_1) =H_2(\bar{y}\upharpoonleft u_2).\]
\end{enumerate}
Next, for $\bar{y}\in\bS_{u_\ell,i+1}$, $\ell=1,2$, put
\[H'_\ell(\bar{y})=\big(g_{y_i}(\alpha_\ell)^{-1}\circ f_{y_i}(\alpha_\ell)
\big)\big(\bh_\ell(\bar{y}\rest i,\bar{y}\upharpoonleft w)\big)\in
{}^{n_*(i)}2.\]  
Apply \ref{3c.26} again (this time using clause ($\delta$) there too) to
choose $\cF''_1,\cF''_2,\bh_1'',\bh_2''$ such that 
\begin{enumerate}
\item[$(*)_4$] $\cF''_\ell\in\wpos^{u_\ell}_i$, $\cF''_\ell\leq \cF_\ell'$, 
  $\|\cF''_\ell\|\geq 8^{-k_*}\cdot \|\cF_\ell'\|$, and the pair
  $(\cF'_1,\cF'_2)$ is balanced, 
\item[$(*)_5$] $\bh''_\ell:\bS_{u_\ell,i}\times\bS_{w,i+1}\longrightarrow
  {}^{n_*(i)}2$, 
\item[$(*)_6$] for each $\bar{x}\in\bS_{u,i}$ one of the following occurs:
\begin{enumerate}
\item[$(\alpha)_{\bar{x}}$] if $h\in\set(\cF''_1*\cF''_2)$ and $\bar{y}=
  \suc_{\bar{x}}(h)\in\bS_{u,i+1}$, then $H'_1(\bar{y}\upharpoonleft u_1)
  \neq H_2'(\bar{y}\upharpoonleft u_2)$, or
\item[$(\beta)_{\bar{x}}$] if $h\in\set(\cF''_1*\cF''_2)$ and $\bar{y}=
  \suc_{\bar{x}}(h)\in\bS_{u,i+1}$, then
\[\bh''_1(\bar{x}\upharpoonleft u_1,\bar{y}\upharpoonleft w)=
\bh''_2(\bar{x}\upharpoonleft u_2,\bar{y}\upharpoonleft w)=
H'_1(\bar{y}\upharpoonleft u_1) =H'_2(\bar{y}\upharpoonleft u_2).\]
\end{enumerate}
\end{enumerate}
It follows from $(*)_1+(*)_4$ that $\frac{|\pos^{u_\ell}_i|}{\|\cF''_\ell\|}
\leq 64^{k_*}\cdot \frac{|\pos^{u_\ell}_i|}{\|\cF_\ell\|}<64^{\ell^*_i}
\cdot \frac{|\pos^{u_\ell}_i|}{\|\cF_\ell\|}$ and hence (remembering that
$\nor^0_i(\cF_\ell)>(\ell^*_i)^2$) we have
\[\begin{array}{l}
\displaystyle
\nor^0_i(\cF''_\ell)\geq k^*_i-\log_3\Big(\log_{k^*_i}\big(\frac{k^*_i\cdot
  |\pos^{u_\ell}_i|}{\|\cF_\ell\|}\cdot 64^{\ell^*_i}\big)\Big)\geq 
k^*_i-\log_3\Big(\log_{k^*_i}\big(\frac{k^*_i\cdot
  |\pos^{u_\ell}_i|}{\|\cF_\ell\|}\cdot k^*_i\big)\Big)\geq \\
\ \\
\displaystyle
k^*_i-\log_3\Big(\log_{k^*_i}\big(\big(\frac{k^*_i\cdot
  |\pos^{u_\ell}_i|}{\|\cF_\ell\|}\big)^3\big)\Big)=
k^*_i-\log_3\Big(3\log_{k^*_i}\big(\frac{k^*_i\cdot
  |\pos^{u_\ell}_i|}{\|\cF_\ell\|}\big)\Big)=\nor^0_i(\cF_\ell)-1>\ell^*_i.
\end{array}\]
In particular,
$\|\cF_\ell''\|/|\pos^{u_\ell}_i|>(k^*_i)^{1-3^{k^*_i-\ell^*_i}}$ and by
\ref{x38}(2) we get 
\[\frac{\|\cF''_1*\cF''_2\|}{|\pos^u_i|}\geq
\Big(\frac{1}{2}(k^*_i)^{1-3^{k^*_i-\ell^*_i}}\Big)^3,\]
so
\begin{enumerate}
\item[$(*)_7$] $\nor^0_i(\cF_1''*\cF_2'')\geq k^*_i- \log_3\big(
  \log_{k^*_i}\big(k^*_i\cdot (2(k^*_i)^{3^{k^*_i-\ell^*_i}-1}
  \big)^3\big)>\ell^*_i-2>0$.  
\end{enumerate} 
Now we claim that 
\begin{enumerate}
\item[$(*)_8$] in clause $(*)_6$ before, the possibility $(\beta)_{\bar{x}}$
  cannot occur.
\end{enumerate}
Suppose towards contradiction that for some $\bar{x}\in\bS_{u,i}$ the
statement in $(\beta)_{\bar{x}}$ holds true. Then, remembering
$\bh_\ell:\bS_{u_\ell,i} \times\bS_{w,i+1}\longrightarrow {}^{n_*(i)}2$, we have
\begin{enumerate}
\item[$(\circledast)$] if $h\in\set(\cF''_1*\cF''_2)$ and
  $\bar{y}=\suc_{\bar{x}}(h)$ and $\eta_\ell=\bh_\ell(\bar{x}\upharpoonleft
  u_\ell, \bar{y}\upharpoonleft w)$ (for $\ell=1,2$),\\
 then $\big(g_{y_i}(\alpha_1)^{-1}\circ f_{y_i}(\alpha_1)\big)(\eta_1) =
  \big(g_{y_i}(\alpha_2)^{-1}\circ f_{y_i}(\alpha_2)\big)(\eta_2)$.
\end{enumerate} 
Since $\alpha_1\neq\alpha_2$ are in $u\setminus w$ we may apply Lemma
\ref{2q.29}(2) to get that $\nor^0_i(\cF''_1*\cF''_2)=0$, contradicting 
$(*)_7$. 

Thus, putting together $(*)_3$ and $(*)_6+(*)_8$ we conclude that 
\begin{enumerate}
\item[$(*)_9$] if $h\in\set(\cF''_1*\cF''_2)$, $\bar{x}\in \bS_{u,i}$ and 
  $\bar{y}=\suc_{\bar{x}}(h)$, $\eta_\ell=H_\ell(\bar{y}\upharpoonleft
  u_\ell)$ (for $\ell=1,2$), then 
\[\eta_1=\eta_2\quad\Rightarrow\quad \big(g_{y_i}(\alpha_1)^{-1}\circ
f_{y_i}(\alpha_1)\big)(\eta_1) \neq\big(g_{y_i}(\alpha_2)^{-1}\circ
f_{y_i}(\alpha_2)\big)(\eta_2).\] 
\end{enumerate} 
Now we set $\gd_\ell=(\cF''_\ell,m_{\gc_\ell})$ (for $\ell=1,2$). Since
$\cF''_\ell\leq \cF'_\ell\leq \cF_\ell$ and $\nor^0_i(\cF''_\ell)\geq
\nor^0_i(\cF_\ell)-1>m_{\gc_\ell}$, we know that $\gd_\ell\in
\underline{\Sigma}(\gc_\ell)$, and since $(\cF''_1,\cF''_2)$ is 
balanced we conclude $\proj_w(\gd_1)=\proj_w(\gd_2)$. Also $\nor_i(\gd_\ell)
\geq \nor_i(\gc_\ell)-1$ and thus $\gd_1,\gd_2$ are as required in
$(\alpha),(\beta)$. Finally, the demand $(\gamma)$ is given by $(*)_9$.
\end{proof}

\begin{lemma}
\label{3c.34} Assume that 
\begin{enumerate}
\item[(a)] $u_1,u_2\subseteq \text{\rm Ord}$ are finite non-empty sets of
  the same size, $|u_1\setminus u_2|=|u_2\setminus u_1|$,
\item[(b)] $w=u_1\cap u_2$, $u=u_1\cup u_2$, and for $\ell=1,2$:
\item[(c)] $p_\ell\in\underline{\bbQ}_{u_\ell}$ and $\alpha_{\ell,k}\in
  u_\ell \setminus w$ and $\name{\rho}_{\ell,k}$ is a
  $\underline{\bbQ}_{u_\ell}$--name for a branch of
  $\name{t}_{\alpha_{\ell,k}}$ (i.e., this is forced) for $k<\omega$, and  
\item[(d)] $\bj_{w,u_1}(p_1),\bj_{w,u_2}(p_2)$ are compatible in
  $\underline{\bbQ}_w$ (see \ref{2q.48}, \ref{2q.45}).   
\end{enumerate}
Then there is $q\in\underline{\bbQ}_u$ such that
$p_\ell\le_{\underline{\bbQ}_{u_\ell}}\bj_{u_\ell,u}(q)$ for $\ell=1,2$ and  
\[q\Vdash_{\underline{\bbQ}_u}\mbox{`` }{\name{\rho}_{1,k}},{\name{\rho}_{2,k}}\mbox{
  have bounded intersection ''}.\]
\end{lemma}

\begin{proof}  
Without loss of generality
\begin{enumerate}
\item[$(\circledast)$] for $\underline{\bbQ}_{u_\ell}$, for each $j<\omega$
  the sequence $\name{\rho}_{\ell,j}$ can be read continuously above
  $p_\ell$; moreover for every large enough $i$, say $i\geq i_\ell(j)$ the 
  sequence $\name{\rho}_{\ell,j} \rest i$ can be read from
  $\name{\bar{x}}_{u_\ell}\rest i$.
\end{enumerate}
[Why? First by Proposition \ref{2q.45} there is $q_1$ such that $p_1
\le_{\underline{\bbQ}_{u_1}} q_1$ and  
\[(\forall q)[q_1\le_{\underline{\bbQ}_{u_1}} q\ \Rightarrow\ \bj_{w,u_1}(q),
\bj_{w,u_2}(p_2)\mbox{ are compatible in }\underline{\bbQ}_w].\] 
Second, by \ref{2q.9}+\ref{2q.43}, there is $p'_1\in\underline{\bbQ}_{u_1}$
satisfying $(\circledast)$ and such that $q_1\le_{\underline{\bbQ}_{u_1}}
p'_1$. Third, we may choose $q_2\ge_{\underline{\bbQ}_{u_2}} p_2$ such that  
\[(\forall q)[q_2 \le_{\underline{\bbQ}_{u_2}} q \Rightarrow \bj_{w,u_1}(p'_1),
\bj_{w,u_2}(q)\mbox{ are compatible in }\underline{\bbQ}_w].\]
Fourth, by \ref{2q.43}, there is $p'_2\in\underline{\bbQ}_{u_2}$ satisfying
$(\circledast)$ and such that $q_2 \le_{\underline{\bbQ}_{u_2}} p'_2$.  Clearly
$(p'_1,p'_2)$ are as required.]\\
Passing to stronger conditions if needed we may also require that
$\bi(p_1)=\bi(p_2)=\bi$, $\bj_{w,u_1}(p_1)=\bj_{w,u_2}(p_2)$ (note $(*)_4$
from the proof of \ref{2q.45}), $|u|<n_*(\bi-1)$ and
$\nor_i(\gc^{p_\ell}_i)>100$ for $i\geq\bi$. Without loss of generality,
letting $i(j)=\max\{i_1(j),i_2(j)\}$, it satisfies $i(0)=\bi$,
$i(j+1)>i(j)+10$ and 
\[\nor_i(\gc_i^{p_1})=\nor_i(\gc_i^{p_2})>2j+2\qquad \mbox{ for }i\geq
i(j).\]   
Fix $i\geq \bi$ for a moment. Let $k$ be such that $i(k)\leq i<i(k+1)$. We
shall shrink $\gc^{p_1}_i,\gc^{p_2}_i$ in order to take care of $(\alpha_{1,m},
\name{\rho}_{1,m},\alpha_{2,m},\name{\rho}_{2,m})$ for $m\leq k$. By
$(\circledast)$ from the beginning of the proof we know that 
\begin{enumerate}
\item[(i)] if $\bar{y}\in \bS_{u_\ell,i+1}\cap \pos(p_\ell)$,\\ 
then the condition $(p_\ell)^{[\bar{y}]}\in\underline{\bbQ}_{u_\ell}$
decides $\name{\rho}_{\ell,m}(i)$ for $m\leq k$, say $(p_\ell)^{[\bar{y}]}
\Vdash_{\underline{\bbQ}_{u_\ell}}$``
$\name{\rho}_{\ell,m}(i)=H_{\ell,m}(\bar{y})$ '', where
$H_{\ell,m}:\bS_{u_\ell,i+1}\longrightarrow {}^{n_*(i)}2$.
\end{enumerate}
Use Crucial Corollary \ref{cruccor} $(k+1)$ times to choose $\gd^1_i \in
\underline{\Sigma}(\gc^{p_1}_i)$ and $\gd^2_i\in
\underline{\Sigma}(\gc^{p_2}_i)$ such that:
\begin{enumerate}
\item[(ii)] $\proj_w(\gd^1_i)=\proj_w(\gd^2_i)$,
\item[(iii)]  $\nor_i(\gd^\ell_i)\geq \nor_i(\gc^{p_\ell}_i)-(k+1)$ (for
  $\ell=1,2$), 
\item[(iv)] {\em if\/} $h\in\set(\cF_{\gd^1_i}*\cF_{\gd^2_i})$, $\bar{x}\in\bS_{u,i}$,
  $\bar{y}=\suc_{\bar{x}}(h)\in \bS_{u,i+1}$, $m\leq k$, $\ell=1,2$ and
  $\eta_\ell=H_{\ell,m}(\bar{y}\upharpoonleft u_\ell)\in {}^{n_*(i)}2$, {\em
    then} 
\[\eta_{1,m}=\eta_{2,m}\quad\Rightarrow\quad \big(g_{y_i}(\alpha_{1,m})^{-1}
\circ f_{y_i}(\alpha_{1,m}))(\eta_{1,m})\neq
(g_{y_i}(\alpha_{2,m})^{-1}\circ f_{y_i}(\alpha_{2,m})) (\eta_{2,m}).\] 
\end{enumerate}
After this construction is carried out for every $i\geq \bi$ we define 
\begin{itemize}
\item $q_\ell=(\bar{x}_{p_\ell},\bar{\gd}^\ell)$, where $\bar{\gd}^\ell
  =\langle \gd^\ell_i: i\in [\bi,\omega)\rangle$, $\ell=1,2$,
\item $q=(\bar{x}_{p_1}\cup \bar{x}_{p_2},\bar{\gd})$, where $\bar{\gd}=
  \langle \gd_i:i\in [\bi,\omega)\rangle$, $\cF_{\gd_i}=\cF_{\gd^1_i}*
  \cF_{\gd^2_i}$, $m_{\gd_i}=m_{\gd^1_i}=m_{\gd^2_i}$.
\end{itemize}
It follows from (iii) (and the choice of $i(j)$) that
$q_\ell\in\underline{\bbQ}_{u_\ell}$ and, by \ref{x38}(2), $q\in
\underline{\bbQ}_u$. Plainly $p_\ell\leq_{\underline{\bbQ}_{u_\ell}} q_\ell
\leq_{\underline{\bbQ}_{u_\ell}} \bj_{u_\ell,u} (q)$. 

Now, let $k<\omega$ and consider $i\geq i(k)$. It follows from
(iv)+\ref{2q.3}(5) that for each $\bar{x}\in\bS_{u,i}\cap\pos(q)$ and $h\in
\set(\cF_{\gd_i})$, if $\bar{y}=\suc_{\bar{x}}(h)$ and $\eta_{\ell,k}=
H_{\ell,k}(\bar{y}\upharpoonleft u_\ell)$, then 
\[\eta_{1,k}=\eta_{2,k}\quad \Rightarrow\quad q^{[\bar{y}]}
\Vdash_{\underline{\bbQ}_u} \mbox{`` }\{\rho:\eta_{1,k}
<_{t_{\name{\varkappa}_{\alpha_{1,k}}}} \rho\}\cap \{\rho:\eta_{2,k}
<_{t_{\name{\varkappa}_{\alpha_{2,k}}}} \rho\}=\emptyset\mbox{ ''}.\]
Since $q^{[\bar{y}]}\Vdash_{\underline{\bbQ}_u}\mbox{`` }
\name{\rho}_{\ell,k}(i)=\eta_{\ell,k}\mbox{ ''}$ (for $\ell=1,2$) we may
conclude that
\[q^{[\bar{y}]}\Vdash_{\underline{\bbQ}_u}\mbox{`` either }
\name{\rho}_{1,k}(i)\neq \name{\rho}_{2,k}(i)\mbox{ or }(\forall j>i)(
\name{\rho}_{1,k}(j)\neq \name{\rho}_{2,k}(j))\mbox{ ''.}\]
Hence immediately we see that $q$ is as required in the assertion of the
lemma. 
\end{proof}

\begin{remark}
\begin{enumerate}  
\item If we can deal only with one case (i.e., one $k$ in clause (c) of
  \ref{3c.34}), we have to use $\cA=\bT^*_\omega$, not ``any uncountable''
  $\cA\subseteq\bT^*_\omega$. But actually it is enough in \ref{3c.34} to
  deal with finitely many pairs. 
\item We can prove in \ref{3c.34} that there is a pair $(p'_1,p'_2)$ such
  that:  
\begin{enumerate}
\item[(a)] $p_\ell\le_{\bbQ_{u_\ell}} p'_\ell$ for $\ell=1,2$,
\item[(b)] $\bj_{w,u_1}(p'_1),\bj_{w,u_2}(p'_2)$ are compatible,
\item[(c)] if $p\in\bbQ_u$ satisfies $p'_\ell\le_{\bbQ_{u_\ell}}
  \bj_{u,u_\ell}(p)$, then $p$ is as required.
\end{enumerate}
\end{enumerate}
\end{remark}

If $u=\{\alpha\}$ is a singleton, then considering $\OB^u_i,\bS_{u,i},
\bS_u,\pos^u_i,\wpos^u_i,\underline{\bbQ}_u$ we may ignore $u$ (and
$\alpha$) in a natural way arriving to the definitions of $\OB_i,\bS_i,
\bS,\pos_i,\wpos_i,\underline{\bbQ}$, respectively. Let
$\varkappa:\bS_\omega\longrightarrow \bT_\omega$ be the mapping given by
$\varkappa(\bar{x})=\langle f_{x_i}:i<\omega\rangle$ (on $\bT$ see
Definition \ref{4d.3}(2), concerning $\varkappa$ compare Definition
\ref{2q.1}(G)). 
\medskip

The following proposition finishes the proof of Theorem \ref{4d.1}. 

\begin{proposition}
\label{3c.37} 
Let $N_*\prec(\cH(\beth^+_7),\theta)$ be countable.
\begin{enumerate}
\item There is a perfect subtree $\bS^*\subseteq\bS$ (so $\bS^*_\omega =
  \lim_\omega(\bS^*)\subseteq\bS_\omega$) such that: 

if $n<\omega$, $\bar{x}_\ell\in\bS^*_\omega$ for $\ell<n$ are pairwise
distinct then $(\bar{x}_0,\ldots,\bar{x}_{n-1})$ is a generic for
$\underline{\bbQ}_n$ over $N_*$.
\item Moreover, $\varkappa[\bS^*_\omega]\subseteq\bT_\omega$ is strongly pbd
  (see Definition \ref{4d.7}(3)) and $\arcl\{A_{\varkappa(\bar{x})}:
  \bar{x}\in \bS^*_\omega\}$ is Borel.  
\end{enumerate}
\end{proposition}

\begin{proof}  
By \ref{2q.43} and \ref{2q.52} and (for part (2)) by \ref{3c.34}.  In
details, let $\cT$ be a perfect subtree of ${}^{\omega>}2$ such that in each
level only in one node we have splitting and let $\cT_i = \{\eta\in
\cT:\eta$ of the $i$-th level$\}$. 

Let $h_i:|\cT_i|\longrightarrow \cT_i$ be a bijection such that 
\[m'<m''< n_i\ \Leftrightarrow\ h_i(m') <_{\text{lex}} h_i(m''),\]
where $n_i=|\cT_i|$. Let $\langle(m_j,k_j,\name{\rho}_j):j<\omega\rangle$
list all the triples $(m,k,\name{\rho})$ satisfying: $m<\omega$, $k<m$ and
$\name{\rho}$ is a $\bbQ_{m \setminus \{k\}}$--name of a branch of
$\name{t}_k$ such that $\name{\rho}$ belongs to $N_*$.

Let $\eta_i$ be the unique member of $\cT_i$ such that $\{\eta_i\conc
\langle 0 \rangle,\eta_i \conc\langle 1 \rangle\}\in \cT_{i+1}$. For
$\ell=0,1$ let  $f_{i,\ell}:\cT_i\longrightarrow \cT_{i+1}$ be such that 
\[[\eta\in \cT_i\setminus \{\eta_i\}\quad \Rightarrow\quad f_{i,\ell}(\eta)
\rest i = \eta]\qquad\mbox{ and }\qquad
f_{i,\ell}(\eta_i)=\eta_i\conc\langle\ell\rangle.\]  
Let $u_{i,\ell}=\text{ Rang}(g_{i,\ell})$ where $g_{i,\ell}=h^{-1}_{i+1}
\circ f_{i,\ell} \circ h_i$.  For an order preserving function $g$ from the
finite $u \subset \text{Ord}$ into $\text{Ord}$ let $\hat{g}$ be the
isomorphism from $\bbQ_u$ onto $\bbQ_{g[u]}$ induced by $g$.

Let $\langle \cI_{n,i}:i<\omega\rangle$ list all the dense open subsets of
$\bbQ_n$ which belong to $N_*$. By induction on $i<\omega$ choose $p_i$ such
that if $\ell \in \{1,2\}$ then (recalling $\bj_{u_{i,\ell},n_j}$ is a
complete projection from $\bbQ_{n_j}$ onto $\bbQ_{u_{i,\ell}}$) we have 
\begin{enumerate}
\item[(i)] $p_i\in\bbQ_{n_i}$, $\hat{g}_{i,\ell}(p_i)
  \le_{\bbQ_{u_{i,\ell}}} \bj_{u_{i,\ell},n_{i+1},}(p_{i+1})$ for
  $\ell=0,1$. 
\item[(ii)] If $u\subseteq n_i$ and $h^*_u$ is $\OP_{u,|u|}$, i.e., the
  order preserving function from $\{0,\ldots,|u|-1\}$ onto $u$, and
  $\hat{h}^*_u$ is defined as above and $k<i$, then $\bj_{u,n_i}(p_i) \in
  \bbQ_u$ belongs to $\hat{h}^*_u(\cI_{|u|,k})$.
\item[(iii)] Assume that for $\ell=0,1$ the objects $j_\ell<\omega$, $u_\ell
  \subseteq \cT_i$ satisfy 
\[\eta_i\in u_\ell,\ |u_\ell|=m_{j_\ell},\ h^*_{u_\ell}(k_{j_\ell})=
  h^{-1}_i(\eta_i)\]
and let $\name{\rho}_\ell=\hat{g}_{i,\ell}(\hat{h}^*_{u_\ell}(
\name{\rho}_{j_\ell}))$ (so it is a $\bbQ_{n_{i+1}}$--name for a branch of
$\name{t}_{g_{i,\ell}(h^*_u(\eta_i))}$). \\
{\em Then\/} $\Vdash_{\bbQ_{n_{i+1}}}$`` the branches $\name{\rho}_0$ of
$\name{t}_{f_{i,0}(\eta_i)}$ and $\name{\rho}_1$ of
$\name{t}_{f_{i,1}(k_{\eta_i})}$ have bounded intersection ''. 
\end{enumerate}
This is straightforward.
\end{proof}

\begin{theorem}
\label{3c.47}
\begin{enumerate}
\item There is a Borel arithmetically closed set $\bB\subseteq \cP(\omega)$
  such that there is no arithmetically closed 2-Ramsey ultrafilter on it.
\item Moreover, there is a Borel\footnote{to eliminate it we have to
   force over $\bbN$} $\cA_*\subseteq \cB$ such that for every uncountable
 $\cA'\subseteq A$, there is no definably closed minimal ultrafilter on the
 arithmetic closure of $\arcl(\cA')$ of $\cA'$.
\item We can demand that above each $\arcl(\cA')$ is a standard system.
\end{enumerate}
\end{theorem}

\begin{proof}
(1) and (2)\quad Let $\cA=\bT^*_\omega$ be as in the proof of \ref{3c.37}
and let $\cB$ be the arithmetic closure $\arcl(\cA)$ of $\cA$. For every
$A_t\in \cA$ there towards contradiction assume $D$ is a $\bB$--minimal
ultrafilter where $\name{B}=\arcl(\cA')$, $\cA' \subseteq \cA$ is
uncountable. 

Now for every $A_t\in\cA'$, $(\bbN,<_t)$ is a tree with finite levels (hence
finite splittings), a root and the set of levels is $\bbN$. For every
$i<\omega$ the set $\{n<\omega$: in $<^*_t$ the level of $n$ is $<i\}$ is
finite and hence its compliment belongs to $D$. The rest is divided to
$\langle\{m:b \le^*_t m\}:b$ is of level exactly $i$ for $<^*_t\rangle$.
This is a finite division hence for some unique $b=b^t_i$ of level $i$ such
that $\{m:b \le^*_t m\} \in D$. As $D$ is a 2-Ramsey ultrafilter
\begin{enumerate}
\item[(i)]  $\langle b^t_i:i < \omega\rangle$ is definable in $\bbN_{\cA'}$.
\end{enumerate}
We define a function $g_t$ on $\bbN$ by $g_t(c)=\max\{i:b^t_i\le_t
c\}$. Again 
\begin{enumerate}
\item[(ii)] $g_t$ is definable in $\bbN_{\cA'}$.
\end{enumerate}
As $D$ is minimal there is $C_t\subseteq\bbN$ definable in $\bbN_{\cA'}$ and
such that
\begin{enumerate}
\item[(iii)] $g_t\rest C_t$ is one-to-one.
\end{enumerate}
Let $C_t$ be the first order definable in $\bbN_{\cA_t}$ where
$\cA_t\subseteq \cA'$ is finite, $t\in \cA_t$ for simplicity and so is the
set $\{b^t_i:i < \omega\}$. As each $\bbQ_u$ is ${}^\omega \omega$--bounding
and we can further shred $c_t$ below there is $h_*\in N_*$ such that [recall
we are forcing over the countable $N_*\prec (H(\chi),\in)$, so our $\cB$ is
$\bigcup\{\cP(\omega)\cap N[t_0,\ldots,t_{n-1}]:t_\ell \in T^*_\omega\}$]
such that 
\begin{enumerate}
\item[(iv)] $h_*\in {}^\omega\omega$ is increasing, $h_*(0)=0$, and 
\item[(v)]  if $c\in C_t$ and $g_t(c)<_t h_*(i)$ then $c<_\bbN h_*(i+1)$.
\end{enumerate}
Without loss of generality now by the infinite $\Delta$--system for finite
sets for some $t_1 \ne t_2$ we have $\{t_1,t_2\}\cap
(\cA_{t_1}\cap\cA_{t_2}) = \emptyset$, etc. 

Moreover, replacing $\cA_{t_1}\cup A_2$, $\cA_1,\cA_2,t_1,t_2$ by $u=u_1\cup
u_2$, $u_1,u_0$, $\alpha_1\in u_1\setminus u_2$, $\alpha_2\in u_2\setminus
u_1$ we have the situation in \S2 by similar proof. We get $C_{t_2}\cap
C_{t_2}$ is finite, but both are in an ultrafilter, so we are done. 
\medskip

(3) We let $\bbQ$ be as in \cite{Sh:F834} for $\lambda \ge
\beth_{\omega_1}$, use what is proved there.   
\end{proof}

%

\end{document}